\documentclass[12pt]{amsart}
\usepackage{url, calc, bm}
\usepackage{amsmath,amsxtra,amssymb,latexsym,epsfig,amscd,amsthm,fancybox,epsfig}
\usepackage[mathscr]{eucal}
\usepackage{graphicx}
\usepackage{multicol,xcolor}
\usepackage{epsfig} % for postscript graphics files
\usepackage{epstopdf}
\usepackage{cases}
\usepackage{subfig}
\usepackage{color}
\usepackage{hyperref}
\usepackage{bigints}
\usepackage{enumerate}
\usepackage{pdfpages}
\graphicspath{{./Figures/}}

\usepackage{bm}

\usepackage[normalem]{ulem}

\usepackage{tikz}
\usetikzlibrary{intersections}

\usepackage{thmtools,thm-restate}

\newtheorem{thm}{Theorem}[section]
\newtheorem {asp}{Assumption}[section]
\newtheorem {conj}{Conjecture}[section]

\newtheorem{rmk}{Remark}[section]

\newtheorem{prop}{Proposition}[section]
\theoremstyle{definition}

\theoremstyle{remark}

\numberwithin{equation}{section}

\newcommand{\F}{\mathcal{F}}

\newcommand{\E}{\mathbb{E}}

\newcommand{\N}{\mathbb{N}}

\newcommand{\PP}{\mathbb{P}}

\newcommand{\R}{\mathbb{R}}

\numberwithin{equation}{section}
\newcommand{\vp}{\varphi}
\newcommand{\UM}{\mathfrak U_{sm}}

\newcommand{\op}{{\mathcal L}}

\newcommand{\bed}{\begin{displaymath}}
\newcommand{\eed}{\end{displaymath}}
\newcommand{\bea}{\bed\begin{array}{rl}}
\newcommand{\eea}{\end{array}\eed}

\newcommand{\barray}{\begin{array}{ll}}
\newcommand{\earray}{\end{array}}

\def\bar{\overline}

\def\a.s{\text{\;a.s.\;}}

\title[Asymptotic harvesting]{Asymptotic harvesting of populations in random environments}

\author[A. Hening]{Alexandru Hening }
\address{Department of Mathematics\\
Tufts University\\
Bromfield-Pearson Hall\\
503 Boston Avenue\\
Medford, MA 02155\\
United States
}
\email{alexandru.hening@tufts.edu}

\author[D. Nguyen]{Dang H. Nguyen }
\thanks{D. Nguyen was in part supported by
 the National Science Foundation
under grant DMS-1207667.}
\address{Department of Mathematics \\
 Wayne State University\\
 Detroit, MI 48202 \\
 United States}
 \email{dangnh.maths@gmail.com}

 \author[S. C. Ungureanu]{Sergiu C. Ungureanu}
\address{Department of Economics\\
City, University of London\\ EC1V 0HB \\
United Kingdom}
\email{Sergiu.Ungureanu.1@city.ac.uk}

  \author[T. K. Wong]{Tak Kwong Wong }
\thanks{T. K. Wong was in part supported by the HKU Seed Fund for Basic Research under the project code 201702159009, and the Start-up Allowance for Croucher Award Recipients.}
\address{Department of Mathematics \\
 The University of Hong Kong\\
 Pokfulam\\
 Hong Kong}
 \email{takkwong@maths.hku.hk}


\begin{document}
\maketitle

\begin{abstract}
We consider the harvesting of a population in a stochastic environment whose dynamics in the absence of harvesting is described by a one dimensional diffusion. Using ergodic optimal control, we find the optimal harvesting strategy which maximizes the asymptotic yield of harvested individuals. To our knowledge, ergodic optimal control has not been used before to study harvesting strategies. However, it is a natural framework because the optimal harvesting strategy will never be such that the population is harvested to extinction -- instead the harvested population converges to a unique invariant probability measure.

When the yield function is the identity, we show that the optimal strategy has a bang-bang property: there exists a threshold $x^*>0$ such that whenever the population is under the threshold the harvesting rate must be zero, whereas when the population is above the threshold the harvesting rate must be at the upper limit. We provide upper and lower bounds on the maximal asymptotic yield, and explore via numerical simulations how the harvesting threshold and the maximal asymptotic yield change with the growth rate, maximal harvesting rate, or the competition rate.

We also show that, if the yield function is $C^2$ and strictly concave, then the optimal harvesting strategy is continuous, whereas when the yield function is convex the optimal strategy is of bang-bang type. This shows that one cannot always expect bang-bang type optimal controls.
\end{abstract}
\tableofcontents

\section{Introduction}

Many species of animals like whales, elephant seals, bisons and rhinoceroses, are at risk of being harvested to extinction (\cite{G71,LB81, LHW93, P06}). Excessive harvesting has already led to both local and gobal extinctions of species (\cite{LES95}). In fact, a significant percentage of the endangered birds and mammals of the world are threatened by harvesting, hunting or other types of overexploitation (\cite{LES95}), and there are similar problems for many species of fish (\cite{hutchings2004marine}). This is why harvesting strategies have to be carefully chosen. After significant harvests, it takes time for the harvested population to get back to the pre-existing level. Moreover, the harvested population fluctuates randomly in time due to \textit{environmental stochasticity}. As a result, an overestimation of the ability of the population to rebound can lead the harvester to overharvest the population to extinction (\cite{LES95}). A less common but nevertheless important problem is an insufficient rate of harvesting. Because of instraspecific competition, the population is bounded in a specific environment, so an extraction rate that is too low would lead to a loss of precious resources. For the same reason, choosing an efficient extraction strategy for valuable species is important (\cite{kokko2001optimal}).

We present a stochastic model of population harvesting and find the \textit{optimal harvesting strategy} that mazimizes the \textit{asmptotic yield} of harvested individuals. We consider a novel framework, the one of \textit{optimal ergodic harvesting}. This is based on the theory of ergodic control (\cite{ABG}). In most stochastic models that exist in the literature, for example \cite{LES95, AS98, LO97}, the population is either assumed to become extinct in finite time, or it can end up being harvested to extinction. In our framework, if the population goes extinct under some harvesting strategy, the asymptotic yield is $0$ and therefore this strategy cannot be optimal. If one wants to ensure that harvested species are preserved, this framework is a natural candidate. Our aim is to present a theory of optimal harvesting that includes the risks of extinction from both environmental noise and harvesting. We assume that the population is homogeneous and can be described by a one dimensional diffusion. The harvesting rate is assumed to be bounded, as infinite harvesting rates would imply an unlimited harvesting capacity, something that is clearly not realistic.

In most cases, environmental noise can be introduced in the system by transforming differential equations into stochastic differential equations (SDE). Such techniques require dealing with significant mathematical difficulties, but their use is not just a case of honoring generality. First, there are direct effects of stochasticity on the predictions of the model, and the parameters quantifying it show up in the results. Second, any realistic biological system will depend on environmental variables that are not, or cannot be, accounted for. The role of stochasticity is to ensure that the solutions proposed are robust to such omissions. For example, if avoiding extinction is important, deterministic models can give misleading solutions even when their parameters are corrected for noise (\cite{smith1978analysis}). The transformation to SDE works especially well when the environmental fluctuations are small and there is no chaos (\cite{LES95}). We focus on models with environmental stochasticity and neglect the \textit{demographic stochasticity} which arises from the randomness of birth and death rates of each indiviual of a population. Throughout the paper we assume that environmental stochasticity mainly affects the growth rate of the population (see \cite{T77, BM77, MBHS78, L81, B02, G88, EHS15, ERSS13, SBA11, HN16} for more details). For computational tractability and for clarity of exposition, we look at a one-dimensional model. Nevertheless, our framework works for any model that can be written as a system of stochastic differential equations  (satisfying some mild assumptions - see \cite{ABG}).

A major limitation of existing models in the literature is the dependence of the optimal solutions on parameters that are hard to quantify. For example, in \cite{LES95} the level at which the population becomes extinct -- the minimal viable population -- must be assumed; without it the yields become infinite. In \cite{AS98} the yield must be time discounted to avoid maximizing over yield infinities, and this requires providing a time value for resources. The minimal viable population is a difficult scientific question (\cite{shaffer1981minimum, traill2007minimum}), and the time value of yields is a difficult economics and policy question, because it implies the comparison of the utility of present and future generations (\cite{dreze1987theory}). In contrast, our model sidesteps the issue by assuming no time preference -- and therefore no bias towards extracting in the present, and resolves the problem of maximizing over infinite yields naturally by looking at asymptotic behavior.

A particular case of our model was studied in \cite{AP81}\footnote{We thank the anonymous referee who has brought the paper to our attention.}. The authors limited themselves to the analysis of harvesting strategies that were of bang-bang type. In \cite{abakuks1979optimal}, one of the co-authors in \cite{AP81} proved that an optimal gathering strategy was necessarily of a bang-bang type in a continuous time Markov chain model, making use of the simplifying assumption of a finite state space. Here, instead, we look at very general possible harvesting strategies in a continuous state stochastic model, and show that the optimal one is of bang-bang type. Our contribution is therefore two-fold. We generalize the setting of \cite{AP81} significantly by looking at very general density-dependent growth rates, not just the logistic case. Moreover, we prove what the authors of \cite{AP81} intuited, namely that the optimal strategy is of bang-bang type; and furthermore that this is true for the larger class of convex yield functions.

Stochastic optimal control applications are common in the finance literature. Following the seminal contributions of \cite{merton1969lifetime, merton1971optimum}, objective functions that are integrals of time discounted instantaneous utility flows are now standard. The crucial simplifying assumption is that of time-additive total utility. The utility flow usually depends on consumption flows, and therefore indirectly on other variables and stochastic constraints. With the time-additive utility assumption, our general yield function can also be interpreted as an instantaneous utility function dependent on yield, and our objective function can be the asymptotic expected utility flow dependent on yield. Because a population stock cannot grow indefinitely in our biological model, we diverge from the general finance literature, where financial returns do not usually depend on the size of the holdings of an individual.

Finally, we generalise a result from one of the stochastic models in \cite{smith1978analysis}, where the equivalent to our yield function has a specific simple form. We show that, when the yield function is weakly convex, the optimal control is bang-bang. However, if the yield function is strictly concave, then the optimal harvesting strategy has to be continuous, in contrast to the bang-bang type optimal strategy we find for a linear yield function. This generalization is useful for economic welfare analysis (a more general form of cost-benefit analysis), which typically relies on a concave utility function, equivalent to the concave yield function herein. In economic models, concavity is assumed to model risk aversion (see \cite[Proposition 6.C.1]{mas1995microeconomic} for justification), and for the convenience of interior solutions to maximisation problems. Concave utility leads to a trade-off between risk and returns in asset choice \cite{merton1971optimum}, so the connection between yield concavity and strategy continuity mentioned above is suggestive of risk management. However, risk management interpretations from the finance literature are not directly applicable here. First, financial asset returns are assumed reasonably to not be decreasing in the asset value owned by investors.\footnote{The assumption may not apply in models with large institutional investors.} Moreover, the risk-return trade-off is captured in models with choice between at least two assets with different risk profiles.\footnote{An ecological model extension that would link this literature to our model would consider optimal extraction policy to maximise a time discounted concave total-yield function when there are at least two populations, situated in different environments with no growth limitation.} If anything, finding a bang-bang optimal strategy when yield is linear is more related to finding corner solutions in maximisation problems with linear utility. A bang-bang strategy uses one of the two extremes of the harvesting rate, depending on the momentary population stock.

The rest of the paper is organized as follows. In Section \ref{s:ergodic} we introduce our model and results. We prove that, if the population in the absence of harvesting survives, the yield function is the identity and the harvesting rate is bounded above by some number $M>0$, then the optimal strategy is always a bang-bang type solution: there exists an $x^*>0$ such that one does not harvest if the current population size lies in the interval $\left[0,x^*\right]$ and harvests at the maximal possible rate, $M$, if the current population size lies in the interval $\left(x^*,\infty \right)$. The proofs of the above results are collected in Appendix \ref{a:proofs}. In Section \ref{s:logistic} we apply our results to the special setting of the logistic Verhulst-Pearl model. In Section \ref{s:cont2} (proofs in Appendix \ref{s:cont}) we show that if the yield function is strictly concave, the optimal harvesting strategy is continuous, and when the yield function is more generally weakly convex, optimal strategy is bang-bang.

Finally, in Section \ref{s:disc} we offer some numerical simulations that show how the optimal harvesting strategies and optimal asymptotic change with respect to the parameters of the model. We also provide a discussion of our results.

\section{Optimal ergodic harvesting}\label{s:ergodic}

We consider a population whose density $\tilde X(t)$ at time $t\geq 0$, in the absence of harvesting, follows the stochastic differential equation (SDE)
\begin{equation}\label{e:logistic}
d\tilde X(t) = \tilde X(t)\mu(\tilde X(t))\,dt + \sigma \tilde X(t)\,dB(t), ~\tilde X(0)=x>0,
\end{equation}
where $(B(t))_{t\geq 0}$ is a standard one dimensional Brownian motion. This describes a population $\tilde X$ with per-capita growth rate given by $\mu(x)>0$ when the density is $\tilde X=x$. The infinitesimal variance of fluctuations in the per-capita growth rate is given by $\sigma^2$.

The following is a standing assumption throughout the paper.
\begin{asp}\label{A:1} The function $\mu:[0,\infty)\to \R$ satisfies:
\begin{itemize}
\item $\mu$ is locally Lipschitz.
\item $\mu$ is decreasing.
\item As $x\to \infty$ we have $\mu(x)\to-\infty$.
\item The function $p(x):=x\mu(x)$ has a unique maximum.
\end{itemize}
\end{asp}

The behavior of \eqref{e:logistic} is not hard to study. In the particular case when $\mu(x)=\bar\mu - \kappa x$ see \cite{EHS15, DP84}. The methods there can be easily adapted to our setting. Alternatively, one could use the general results from \cite{HN16}.
The process $\tilde X$ does not reach $0$ or $\infty$ in finite time and the stochastic growth rate $\mu(0)-\frac{\sigma^2}{2}$
determines the long-term behavior in the following way:
\begin{itemize}
\item If $\mu(0)-\frac{\sigma^2}{2}>0$ and $ \tilde X(0)=x>0$, then $(\tilde X(t))_{t\geq 0}$ converges weakly to its unique invariant probability measure $\nu$ on $(0,\infty)$.
\item If $\mu(0)-\frac{\sigma^2}{2}<0$ and $ \tilde X(0)=x>0$, then $\lim_{t\to \infty} \tilde X(t)=0$ almost surely.
\end{itemize}
We let $\R_+:=[0,\infty)$ and $\R_{++}:=(0,\infty)$ throughout the paper.

Assume that the population is harvested at time $t\geq 0$ at the \textit{stochastic rate} $h(t)\in U:=[0,M]$ for some fixed $M>0$. Adding the harvesting to \eqref{e:logistic} yields the SDE
\begin{equation}\label{e:con-diff}
dX(t) = X(t)(\mu (X(t)) -h(t) )\,dt + \sigma X(t)\,dB(t), ~X(0)=x>0.
\end{equation}

A stochastic process $(h(t))_{t\geq 0}$ taking values in $U$
is said to be an \textit{admissible strategy} if
$(h(t))_{t\geq 0}$ is adapted to the filtration $(\F_t)_{t\geq 0}$ generated by the Brownian motion $(B(t))_{t\geq 0}$.
Let $\mathfrak{U}$ be the class of admissible strategies.
An important subset of $\mathfrak{U}$ is the
class $\UM$ of \textit{stationary Markov strategies},
that is, admissible strategies of the form
$h(t)=v(X(t))$ where $v:\R_{++}\mapsto U$ is a measurable function. By abuse
of terminology, we often refer to the map $v(\cdot)$ as the stationary Markov strategy.
Using a stationary Markov strategy $v(\cdot)$,
\eqref{e:con-diff} becomes
\begin{equation}\label{e:c-harvest}
dX(t) = X(t)(\mu(X(t)) -v(X(t)) )\,dt + \sigma X(t)\,dB(t), ~X(0)=x>0.
\end{equation}

\begin{rmk}
The sigma algebra $\F_t$ gives one the information available from time $0$ to time $t$. An admissible harvesting strategy is therefore a strategy which can take into account all the information from the start of the harvesting to the present. These strategies are much more general than constant strategies. Stationary Markov strategies are the harvesting strategies which only depend on the present state of the population density.
\end{rmk}
We associate with $X(t)$ the family of generators $(\op_u)_{u\in[0,M]}$ defined by their action on $C^2$ functions with compact support in $\R_{++}$ as
\begin{equation}\label{e:gen_u}
\op_u f(x):=x[\mu(x)-u]f_x+\dfrac12\sigma^2 x^2f_{xx}.
\end{equation}
We will call $\Phi:\R_+\to \R_+$ a \textit{yield function} if the following assumption holds.
\begin{asp}\label{A:2} The function $\Phi:\R_+\to \R_+$ satisfies:
\begin{itemize}
\item $\Phi$ is continuous.
\item $\Phi(0)=0$.
\item $\Phi$ has subpolynomial growth that is, there is $n \in \mathbb{N}$ such that $\frac{\Phi(x)}{x^n} \to 0$ for $x\to \infty$.
\end{itemize}
\end{asp}
Our aim is to find the optimal strategy $h(t)$ that almost surely maximizes the \textit{asymptotic yield}
\begin{equation}\label{opt-har}
\liminf_{T\to\infty}\dfrac1T\int_0^T \Phi\Big(X(t)h(t)\Big)\,dt.
\end{equation}

In other words we want to find $v$ such that, for any initial population size $X(0)=x>0$, we have with probability $1$ that
\[
\liminf_{T\to\infty}\dfrac1T\int_0^T \Phi\Big(X(t)v(X(t))\Big)\,dt = \sup_{h\in \mathfrak{U}} \liminf_{T\to\infty}\dfrac1T\int_0^T \Phi\Big(X(t)h(t)\Big)\,dt=:\rho^*.
\]
We note that many of the existing models that look at the optimal harvesting of a population in a stochastic environment (\cite{LO97, AS98, LES95}) assume that the yield function $\Phi$ is the identity i.e. $\Phi(x)=x, x\geq 0$. This assumption is not always justifiable (see \cite{A00}) and as such we present in Section \ref{s:cont2} results for more general functions $\Phi$.
\begin{rmk}\label{r:ext}
We note that if $X$ has an invariant probability measure $\pi$ on $\R_{++}$, then for any $X(0)=x>0$ almost surely
\[
\lim_{T\to\infty}\dfrac1T\int_0^T \Phi\Big(X(t)v(X(t))\Big)\,dt = \int_{\R_{++}} \Phi(xv(x))\pi(dx).
\]
In particular, if $X$ goes \textit{extinct}, that is, for any $X(0)=x>0$ we have with probability $1$
\[
\lim_{t\to\infty}X(t)=0,
\]
then the only invariant ergodic measure of $X$ on $\R_+$ is $\delta_0$ the point mass at $0$, and hence, we get that with probability $1$
\[
\lim_{T\to\infty}\dfrac1T\int_0^T \Phi\Big(X(t)v(X(t))\Big)\,dt =0.
\]
Our method for maximizing the asymptotic yield forces the optimal harvesting to be such that the population persists.
\end{rmk}

\begin{rmk}
By \cite[Theorems 2.2.2 and 2.2.12]{ABG},
the controlled systems \eqref{e:con-diff} and \eqref{e:c-harvest} have unique local solutions on $\R_{++}$
for any admissible control $h(t)$ and stationary Markov control $v$ respectively.
Note that one can find $N>0$ large enough such that
$$
\begin{aligned}
\op_u\left(x+\frac1x\right)=&x(\mu(x)-u)\left(\frac{x^2-1}{x^2}\right)+ \sigma^2x^2\frac{1}{x^3}\\
\leq& N(\sigma^2+M)\left(x+\frac1x\right),\, x\in\R_{++}, u\in U.
\end{aligned}
$$
With this fact in hand, we can use the arguments from \cite[Theorem 3.5]{RK} to obtain
the existence of global solutions on $\R_{++}$ of \eqref{e:con-diff} and \eqref{e:c-harvest}.
In particular we get that
\[
\PP_x\left(X(t)\in \R_{++}, t\geq 0\right)=1, x\in\R_{++}.
\]
\end{rmk}

The main result of the paper is the following.
\begin{restatable}{thm}{main}\label{t:main}
Assume that $\Phi(x)=x, x\in(0,\infty)$ and that the population survives in the absence of harvesting, that is $\mu(0)-\frac{\sigma^2}{2}>0$. Furthermore assume that the drift function $\mu(\cdot)$ satisfies Assumption \ref{A:1}. The optimal control (the optimal harvesting strategy) $v$ has the bang-bang form
\begin{equation}\label{e:v}
\begin{aligned}
v(x)
&= \begin{cases}
0 & \mbox{if $0< x\leq x^*$} \\
M & \mbox{if $x>x^*$}
\end{cases}
\end{aligned}
\end{equation}
for some $x^*\in (0,\infty)$.
Furthermore, we have the following upper bound for the optimal asymptotic yield
\begin{equation}\label{e:yield_bound}
\rho^* \leq \sup_{x\in \R_+} x\mu(x).
\end{equation}
\end{restatable}

\section{Continuous vs bang-bang optimal harvesting strategies}\label{s:cont2}

As showcased in Theorem \ref{t:main}, when $\Phi$ is the identity function the optimal harvesting strategy is of bang-bang type. In Appendix \ref{s:cont} we prove the following result.

\begin{restatable}{thm}{cont}\label{t:cont}
Suppose Assumption \ref{A:1} holds and the yield function satisfies
\begin{enumerate}
\item $\Phi\in C^2(\R_+)$,
\item $\Phi$ is strictly concave.
\end{enumerate}
Then the optimal harvesting strategy is \textbf{continuous} and given by
\[\begin{aligned}
v &= \begin{cases}
0 &\mbox{if $[\Phi']^{-1}(V_x^*(x))\leq 0$},\\
\displaystyle\frac{[\Phi']^{-1}(V_x^*(x))}{x} &\mbox{if $0<[\Phi']^{-1}(V_x^*(x))<xM$},\\
M &\mbox{if $[\Phi']^{-1}(V_x^*(x))\geq xM$}.
\end{cases} \\
\end{aligned}\]
Furthermore, the HJB equation for the system becomes
\begin{equation}\label{e:HJB_1d_2}
\begin{split}
\rho  &= \begin{cases}
x\mu(x)f_x+\dfrac12\sigma^2 x^2f_{xx} & \mbox{if $[\Phi']^{-1}(f_x(x))\leq 0$},\\
x\mu(x)f_x+\dfrac12\sigma^2 x^2f_{xx} -f_x[\Phi']^{-1}(f_x) + \Phi([\Phi']^{-1}(f_x)) &\mbox{if $0<[\Phi']^{-1}(f_x(x))<xM$},\\
x(\mu(x)-M)f_x+\dfrac12\sigma^2 x^2f_{xx}+\Phi(xM) &\mbox{if $[\Phi']^{-1}(f_x(x))\geq xM$.}
\end{cases}
\end{split}
\end{equation}
\end{restatable}

\begin{rmk}
We cannot find the exact form of the optimal harvesting strategies in this case. Note that in Theorem \ref{t:main} we have $\Phi(x)=x$ which is not strictly concave nor strictly convex.
\end{rmk}

Intuitively, this is not unlike maximising a strictly concave objective function under a linear constraint. The optimal choice usually moves smoothly over the domain as the direction of the constraint changes. However, when the objective function is weakly convex, e.g. linear, the optimum will jump on the allowed interval.
\newline
\indent Here, we show that if the yield function is weakly convex, the optimal control is bang-bang. The optimal strategy has a similar form to the one for linear yield, if a further assumption on the joint rates of change of the population growth rate and the yield function is made.
\begin{restatable}{thm}{mainvex}\label{t:mainvex}
Assume that $\Phi:\R_+\to\R_+$ is weakly convex, $\Phi$ grows at most polynomially, $\Phi \in C^1(\mathbb{R_+})$ and the population survives in the absence of harvesting, that is $\mu(0)-\frac{\sigma^2}{2}>0$. Furthermore assume that the drift function $\mu(\cdot)$ satisfies the following modification of Assumption \ref{A:1}:
\begin{itemize}
\item[(i)] $\mu$ is locally Lipschitz.
\item[(ii)] $\mu$ is decreasing.
\item[(iii)] As $x\to \infty$ we have $\mu(x)\to-\infty$.
\item[(iv)] The function
\begin{equation}\label{e:G}
G(x)= \Phi(xM)\left(1-\frac{2}{\sigma^2}\mu(x)\right)- xM\Phi'(xM)
\end{equation}
has a unique extreme point in $(0,\infty)$ which is a minimum.
\end{itemize}
If the assumptions (i)-(iii) hold, the optimal control has a bang-bang form (i.e., the harvesting rate is either 0 or the maximal $M$). If assumptions (i)-(iv) hold, the optimal harvesting strategy $v$ has a bang-bang form with one threshold
\begin{equation*}
\begin{aligned}
v(x)
&= \begin{cases}
0 & \mbox{if $0< x\leq x^*$} \\
M & \mbox{if $x>x^*$}
\end{cases}
\end{aligned}
\end{equation*}
for some $x^*\in (0,\infty)$.
\end{restatable}

\section{The logistic case: $\mu(x)=\bar \mu -\kappa x$}\label{s:logistic}

Throughout this section we provide a thorough analysis of the logistic Verhulst-Pearl model. As such, we will assume that the growth rate is $\mu(x)=\bar \mu -\kappa x$ for positive constants $\bar\mu, \kappa>0$. It is clear that this $\mu(\cdot)$ satisfies Assumption \ref{A:1}.
If we harvest according to a \textit{constant strategy} $\ell>0$ then the SDE \eqref{e:c-harvest} becomes
$$dX(t) = X(t)(\bar \mu - \kappa X(t) -\ell)\,dt + \sigma X(t)\,dB(t).$$
It is then easy to see that, as long as $\bar\mu-\ell-\frac{\sigma^2}{2}>0$, the asymptotic yield is
\[
L(\ell):=\lim_{T\to\infty}\dfrac1T\int_0^T \ell X(t)\,dt = \ell \frac{\bar\mu-\ell-\frac{\sigma^2}{2}}{\kappa}.
\]
We can maximize this yield $L(\ell)$, which is quadratic in $\ell$. The maximum will be at
\[
\ell^* = \frac{1}{2}\left(\bar\mu-\frac{\sigma^2}{2}\right)
\]
and the maximal asymptotic yield (among constant harvesting strategies) is
\[
L(\ell^*)= \frac{\left(\bar\mu-\frac{\sigma^2}{2}\right)^2}{4\kappa}.
\]
Note that $L(\ell^*)$ is also called \textit{maximum sustainable yield (MSY)} in the literature.
Since $x\mu(x)=\bar\mu x-\kappa x^2$ we note that
\[
\sup_{x\in \R_+} x\mu(x)=\frac{\bar\mu^2}{4\kappa}.
\]
Combining this with \eqref{e:yield_bound} one sees that the optimal asymptotic yield $\rho^*$ satisfies
\[
\frac{\left(\bar\mu-\frac{\sigma^2}{2}\right)^2}{4\kappa}\leq \rho^*\leq \frac{\bar\mu^2}{4\kappa}.
\]

 Note that Theorem \ref{t:main} does not give us information about $x^*$, the point at which one starts harvesting.

One possible strategy to find out more information about $x^*$ is the following: Look at controls of bang-bang type that have a threshold at $\eta$ and then maximize over all possible $\eta$. This will then give us a way of finding $x^*$.
Let $w(x;\eta)$ be the harvesting strategy
\begin{equation}\label{e:w}
\begin{aligned}
w(x;\eta)
&= \begin{cases}
0 & \mbox{if $0< x\leq \eta$} \\
M & \mbox{if $x>\eta$}.
\end{cases}
\end{aligned}
\end{equation}
For this control $w$ our diffusion \eqref{e:c-harvest} (with $h\equiv w$) is of the form
\begin{equation}\label{e:XX}
dX(t) = a(X(t))\,dt + b(X(t))\,dB(t)
\end{equation}
for
\[
a(x)= x(\bar\mu-w(x,\eta)-\kappa x)
\]
and
\[
b(x)=\sigma x.
\]
Standard diffusion theory shows (see \cite{HK16, BS12}) that the boundary $0$ is natural and the boundary $\infty$ is entrance for the process $X$ from \eqref{e:XX}. As a result, when $\mu-\frac{\sigma^2}{2}>0$, one can show using \cite{BS12} that the density $\rho:(0,\infty)\to (0,\infty)$ of the invariant measure $\pi$ is of the form
\begin{equation}\label{e:rho}
\begin{aligned}
\rho(y) &= \frac{C_1}{b^2(y)}\exp\left(2\int_{\eta}^y \frac{a(z)}{b^2(z)}\,dz\right)\\
&=\frac{C_1}{\sigma^2 y^2} \exp\left(2\int_{\eta}^y \frac{z(\bar\mu-w(z,\eta)-\kappa z)}{\sigma^2 z^2}\,dz\right)\\
 &= \begin{cases}
\frac{C_1}{\sigma^2 y^2}\left(\frac{y}{\eta}\right)^{\frac{2\bar\mu}{\sigma^2}}e^{-\frac{2\kappa}{\sigma^2}(y-\eta)} & \mbox{if $0< y\leq \eta$} \\
\frac{C_1}{\sigma^2 y^2}\left(\frac{y}{\eta}\right)^{\frac{2(\bar\mu-M)}{\sigma^2}}e^{-\frac{2\kappa}{\sigma^2}(y-\eta)} & \mbox{if $y>\eta$},
\end{cases}
\end{aligned}
\end{equation}
where $C_1$ is a normalizing constant given by
\[
\frac{1}{C_1} = \int_0^{\eta}\frac{1}{\sigma^2 y^2}\left(\frac{y}{\eta}\right)^{\frac{2\bar\mu}{\sigma^2}}e^{-\frac{2\kappa}{\sigma^2}(y-\eta)}\,dy + \int_{\eta}^\infty\frac{1}{\sigma^2 y^2}\left(\frac{y}{\eta}\right)^{\frac{2(\bar\mu-M)}{\sigma^2}}e^{-\frac{2\kappa}{\sigma^2}(y-\eta)}\,dy.
\]
In this case the harvesting yield is
\begin{equation}\label{e:yield}
\begin{aligned}
H(\eta)&:=\lim_{T\to\infty}\dfrac1T\int_0^T \Phi\Big(X(t)w(X(t),\eta)\Big)\,dt \\
&= \int_{\R_{++}} y w(y,\eta)\pi(dy) \\
&= \int_0^\infty y w(y,\eta)\rho(y)dy\\
&= \int_{\eta}^\infty yM\frac{C_1}{\sigma^2 y^2}\left(\frac{y}{\eta}\right)^{\frac{2(\bar\mu-M)}{\sigma^2}}e^{-\frac{2\kappa}{\sigma^2}(y-\eta)}\,dy\\
&= \frac{\displaystyle \int_{\eta}^\infty yM\frac{1}{\sigma^2 y^2}\left(\frac{y}{\eta}\right)^{\frac{2(\bar\mu-M)}{\sigma^2}}e^{-\frac{2\kappa}{\sigma^2}(y-\eta)}\,dy}{\displaystyle \int_0^{\eta}\frac{1}{\sigma^2 y^2}\left(\frac{y}{\eta}\right)^{\frac{2\bar\mu}{\sigma^2}}e^{-\frac{2\kappa}{\sigma^2}(y-\eta)}\,dy + \displaystyle \int_{\eta}^\infty\frac{1}{\sigma^2 y^2}\left(\frac{y}{\eta}\right)^{\frac{2(\bar\mu-M)}{\sigma^2}}e^{-\frac{2\kappa}{\sigma^2}(y-\eta)}\,dy}
\end{aligned}
\end{equation}

%%%% Figure of H(x) for conjecture, sigma=mu=kappa=M=1
\begin{figure}[h]
\begin{center}
\includegraphics[scale=0.8]{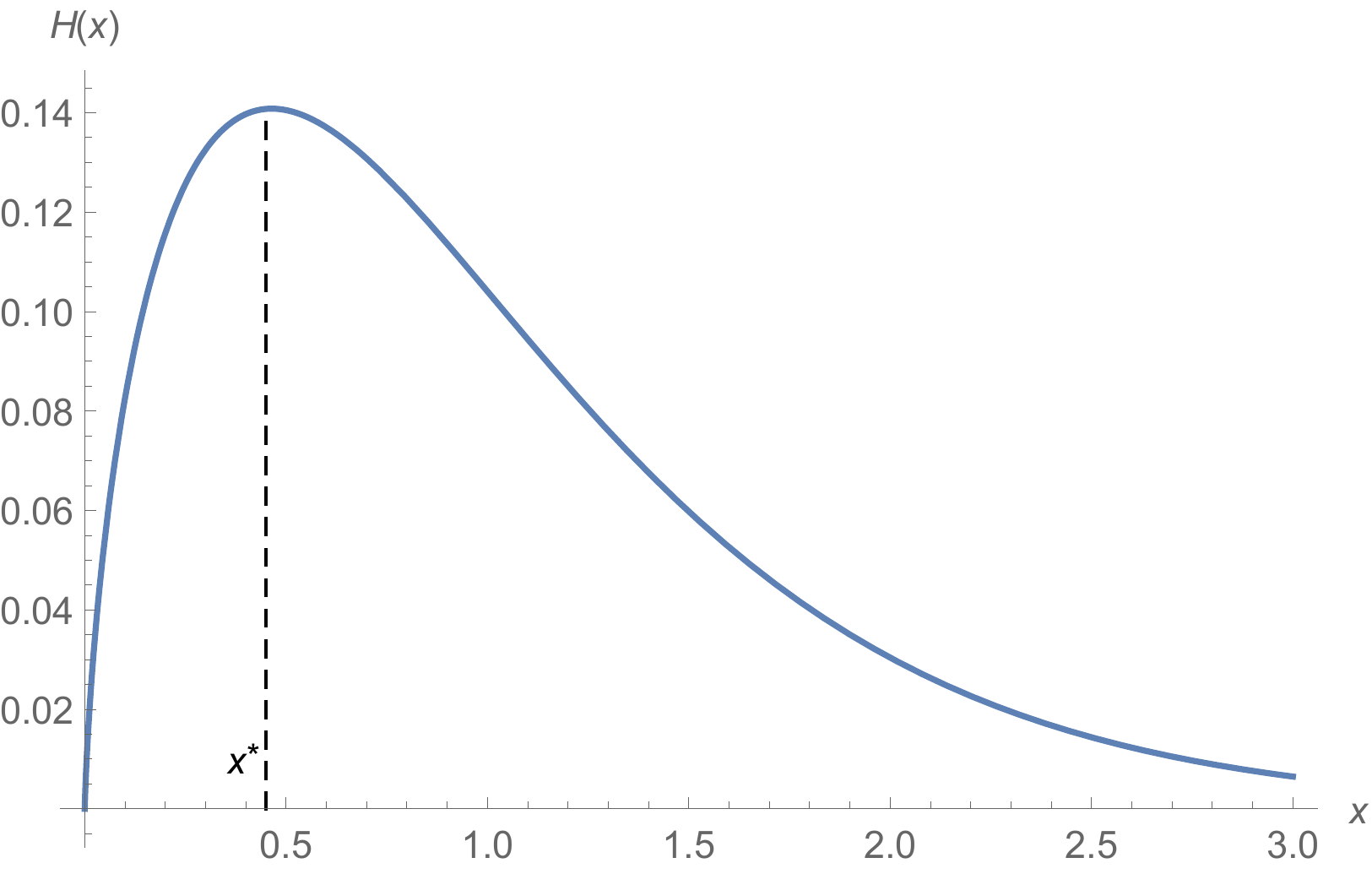}
\caption{Typical shape of the asymptotic yield function $H(x)$ as a function of the harvesting threshold $x$, where one begins to harvest. Here for $\sigma^2=1$ and $M=\bar\mu=\kappa=1$.}
\label{fig:conjexample}

\end{center}
\end{figure}
%%%%%%%%%%%%%%%%%%%%

By Theorem \ref{t:main} the point $x^*$ has to satisfy:
\[
H(x^*)= \max_{\eta\in (0,\infty)}H(\eta).
\]
It is clear that $H$ is differentiable, that $x^*$ exists and satisfies $x^*\in (0,\infty)$.
Therefore, $x^*$ is a solution of
\begin{equation}\label{e:H'}
H'(\eta)=0.
\end{equation}
The condition above can be restated as an equation involving incomplete gamma functions. We were not able to prove analytically that \eqref{e:H'} has a unique solution. \cite{berg2006chen, berg2008convexity} show possible analytical methods that can be applied to such equations in a simple case. However, numerical experiments that we have done support this conjecture (see Figure \ref{fig:conjexample}).

\begin{conj}
There exists a unique $x^*\in (0,\infty)$ such that $H'(x^*)=0$. Furthermore, the optimal harvesting strategy is given by
\begin{equation*}
\begin{aligned}
v(x)
&= \begin{cases}
0 & \mbox{if $0< x\leq x^*$} \\
M & \mbox{if $x>x^*$}.
\end{cases}
\end{aligned}
\end{equation*}
\end{conj}

\section{Discussion and future research}\label{s:disc}

We have analysed a population whose dynamics evolves according to generalization of the logistic Verhulst-Pearl model in a stochastic environment, but subjected to strategic harvesting. The rate at which the population gets harvested is bounded above by a constant $M>0$, and the harvested infinitesimal amount is proportional to the current size of the population.
We show that the harvesting strategy $v$, which describes the harvesting rate and is chosen to maximize the asymptotic harvesting yield
\[
\liminf_{T\to\infty}\dfrac1T\int_0^T X(t)h(t)\,dt,
\]
is of bang-bang type, i.e. there exists $x^*>0$ such that
\begin{equation*}
\begin{aligned}
v(x)
&= \begin{cases}
0 & \mbox{if $0< x\leq x^*$} \\
M & \mbox{if $x>x^*$}.
\end{cases}
\end{aligned}
\end{equation*}

%%%%% Figure of H(x) where \mu varies %%%
\begin{figure}
\begin{center}
\includegraphics[scale=0.8]{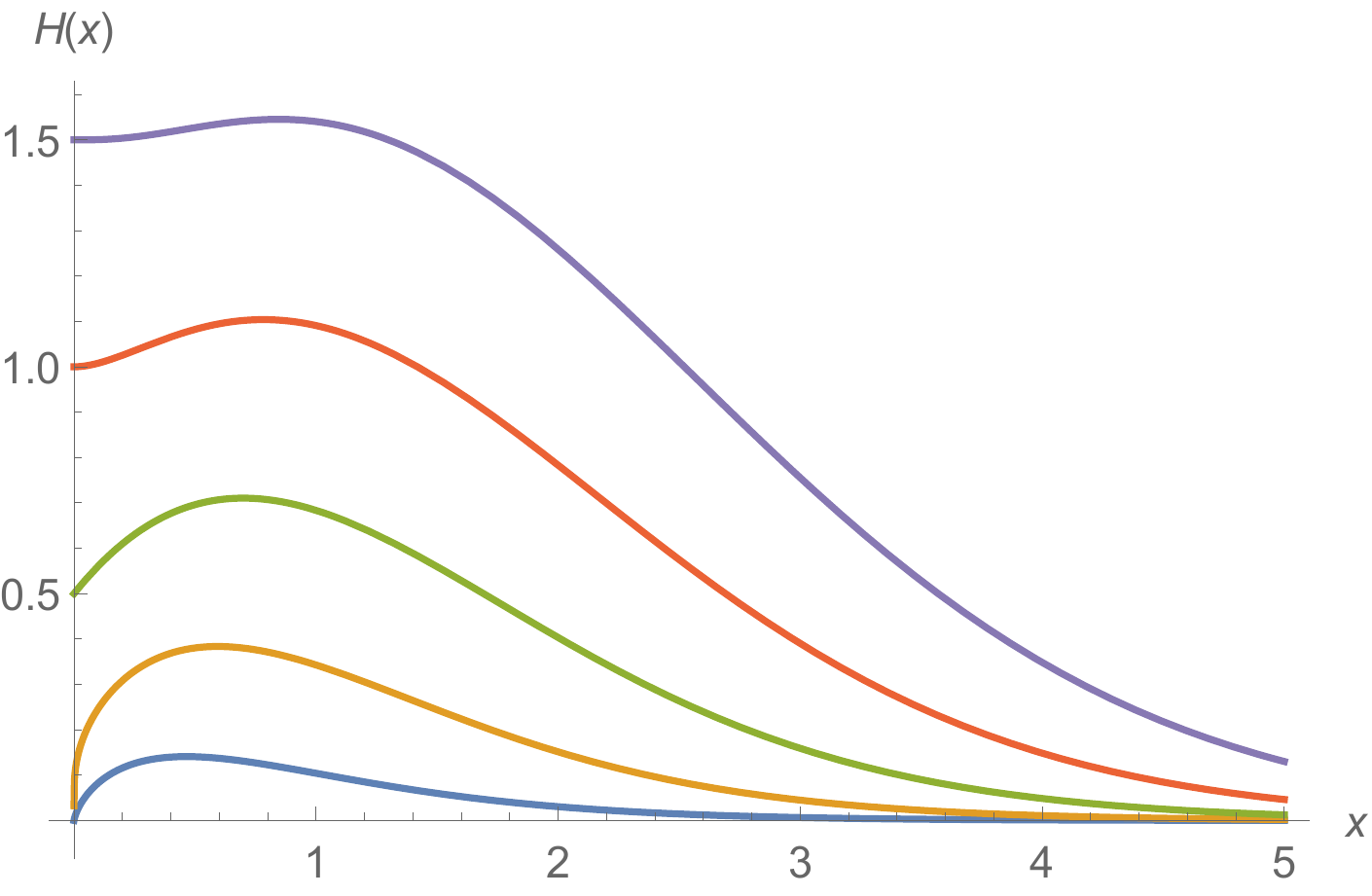}
\caption{Graph of the maximal asymptotic yield $H(x)$ as a function of the harvesting threshold $x$ for different values of the growth rate $\bar\mu$. We take $\sigma^2=1$, $M=\kappa=1$, and $\bar\mu=1$ (blue), $\bar\mu=1.5$ (orange), $\bar\mu=2$ (green), $\bar\mu=2.5$ (red), and $\bar\mu=3$ (purple).}
\label{fig:varymu}

\end{center}
\end{figure}
%%%%%%%%%%%%%%%%%%%%

\subsection{Logistic Verhulst-Pearl } In the particular case when $\mu(x)=\bar\mu - \kappa x$, we can give more information about $x^*$ as follows:
The harvesting yield function $H(\eta)$ is determined, by letting the jump in the bang-bang control be at $\eta$. That is, we look at the yield when the control is
\begin{equation*}
\begin{aligned}
v(x)
&= \begin{cases}
0 & \mbox{if $0< x\leq \eta$} \\
M & \mbox{if $x>\eta$}.
\end{cases}
\end{aligned}
\end{equation*}

The typical behavior of the point $x^*$ where $H$ is maximized and of $H(x^*)$ as the parameters $\bar\mu, \kappa$ and $M$ change was analyzed numerically and is presented in Figures \ref{fig:varymu}, \ref{fig:varym} and \ref{fig:varykappa}, with the normalization $\sigma^2=1$.

We note from numerical experiments that increasing the growth rate $\bar\mu$ increases the threshold $x^*$ at which one should start optimally harvesting (Figure \ref{fig:varymu}). This is an intuitive result, since an increased growth rate increases the maximal equilibrium value of the population in the equivalent deterministic growth model with competition (\cite{smith1978analysis}). Therefore, it should also increase asymptotic harvesting yield, as well as the point at which harvesting should start. Moreover, higher growth rates make the population get faster to the point $x^*$ where one starts harvesting, reducing the cost of a delay.

If one increases the maximal harvesting rate $M$ then the harvesting threshold $x^*$ is also increased (Figure \ref{fig:varym}). This also makes sense because if $\bar\mu-\frac{\sigma^2}{2}-M<0$, then a population with constant harvesting rate $M$ will go extinct almost surely. An increase in the harvesting threshold $x^*$ is necessary to make sure that there is no extinction. Moreover, as $M$ gets larger one can wait longer to start harvesting. With a larger maximal rate available, there is less chance that there will be losses because the population overshoots the optimal extraction point. Similarly, increasing the harvesting rate $M$ also increases the maximal asymptotic harvesting yield, for the obvious reason that there is better control on the population level and therefore extraction can happen closer to the optimal level.

In contrast, if one increases the intraspecific competition rate $\kappa$, then the harvesting threshold decreases (Figure \ref{fig:varykappa}). The equilibrium value of the population in the equivalent deterministic model (\cite{smith1978analysis}) decreases with $\kappa$, and as a result so does the extraction rate. Evidently, even in the stochastic model, if competition is very strong the population cannot spend much time at high densities, and therefore one has to start harvesting early. An increase in $\kappa$ will also decrease the maximal asymptotic harvesting yield.

%%%%% Figure of H(x) where M varies %%%
\begin{figure}
\begin{center}
\includegraphics[scale=0.8]{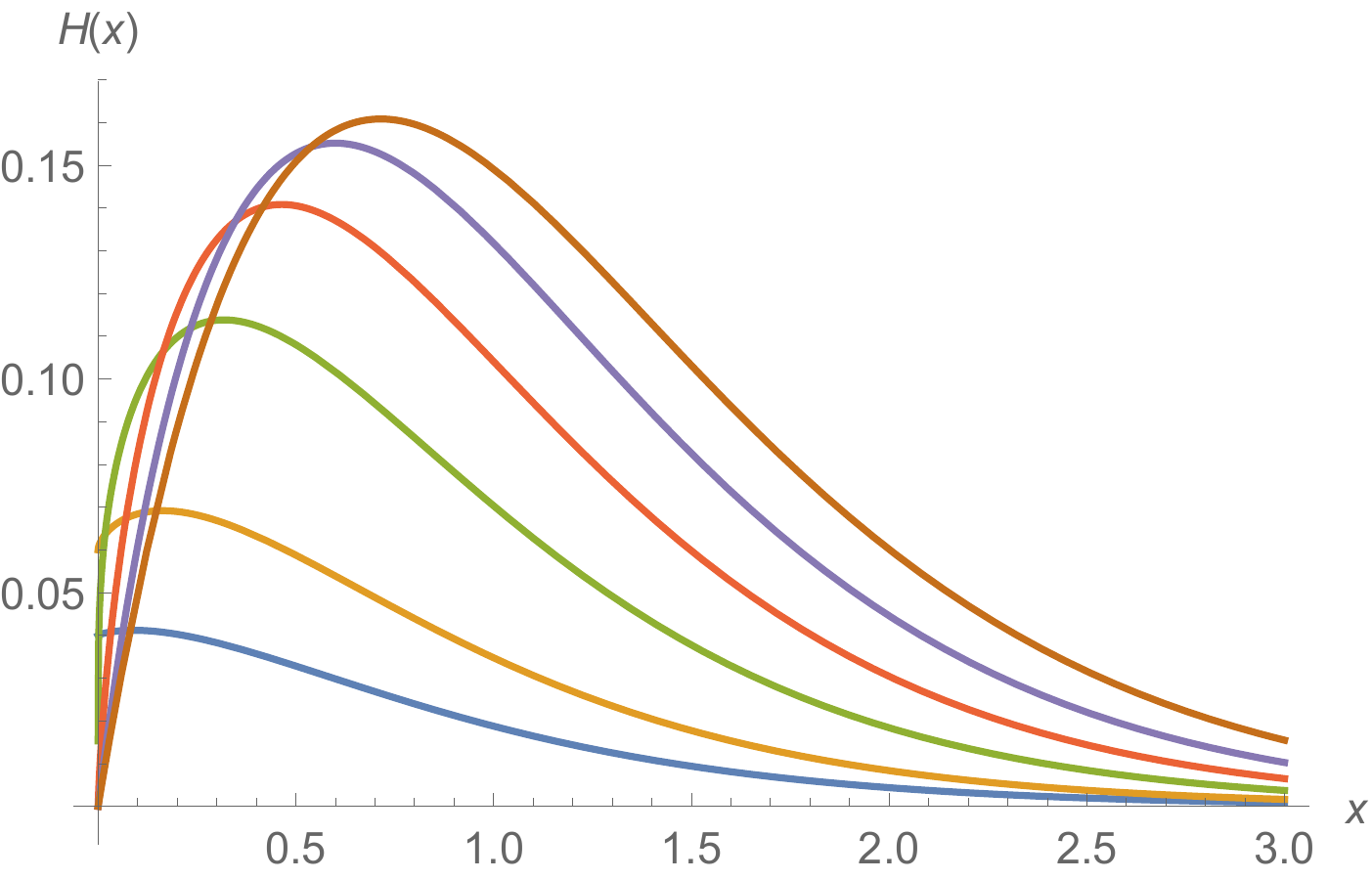}
\caption{Graph of the maximal asymptotic yield $H(x)$ as a function of the harvesting threshold $x$ for different values of the maximal harvesting rate $M$. We take $\sigma^2=1$, $\bar\mu=\kappa=1$, and $M=0.1$ (blue), $M=0.2$ (orange), $M=0.5$ (green), $M=1$ (red), $M=2$ (purple), and $M=5$ (brown).}
\label{fig:varym}

\end{center}
\end{figure}
%%%%%%%%%%%%%%%%%%%%

We are able to prove that the maximal asymptotic yield $\rho^*$ satisfies the inequality
\[
\frac{\left(\bar\mu-\frac{\sigma^2}{2}\right)^2}{4\kappa}\leq \rho^*\leq \frac{\bar\mu^2}{4\kappa}.
\]
In particular, the bang-bang optimal strategy has a higher asymptotic yield than the optimal constant harvesting strategy. Moreover, the bang-bang optimal strategy gives a lower asymptotic yield than the optimal constant harvesting strategy in the absence of noise. This means that the analysis of the more complex stochastic model was fruitful, recommending a qualitatively different strategy. Moreover, environmental fluctuations decrease the maximal asymptotic yield and, because the correction is negative, protecting a population from extinction requires a careful measurement of natural fluctuations when designing optimal harvesting. When environmental stochasticity was not taken into account, harvesting often lead populations to extinction (\cite{LES95}).

Real populations do not evolve in isolation. As a result, ecology is concerned with understanding the characteristics that allow species to coexist. Harvesting can disturb the coexistence of species. In future research we intend to tackle multi-dimensional analogues of the setting treated in the current article. Natural models for which one can add harvesting would be predator-prey food chains (\cite{GH79,G84, HN17b, HN17, TL16}), more general Kolmogorov systems (\cite{SBA11, HN16}) and structured populations where there can be asymmetric harvesting (\cite{ERSS13, EHS15, HNY17, RS14, BS09, schreiber11}). In the multi-dimensional setting the Hamilton-Jacobi-Bellman (HJB) equation becomes a PDE and the analysis becomes significantly more complex. New tools will have to be developed to tackle these problems.

%%%%% Figure of H(x) where \kappa varies %%%
\begin{figure}
\begin{center}
\includegraphics[scale=0.8]{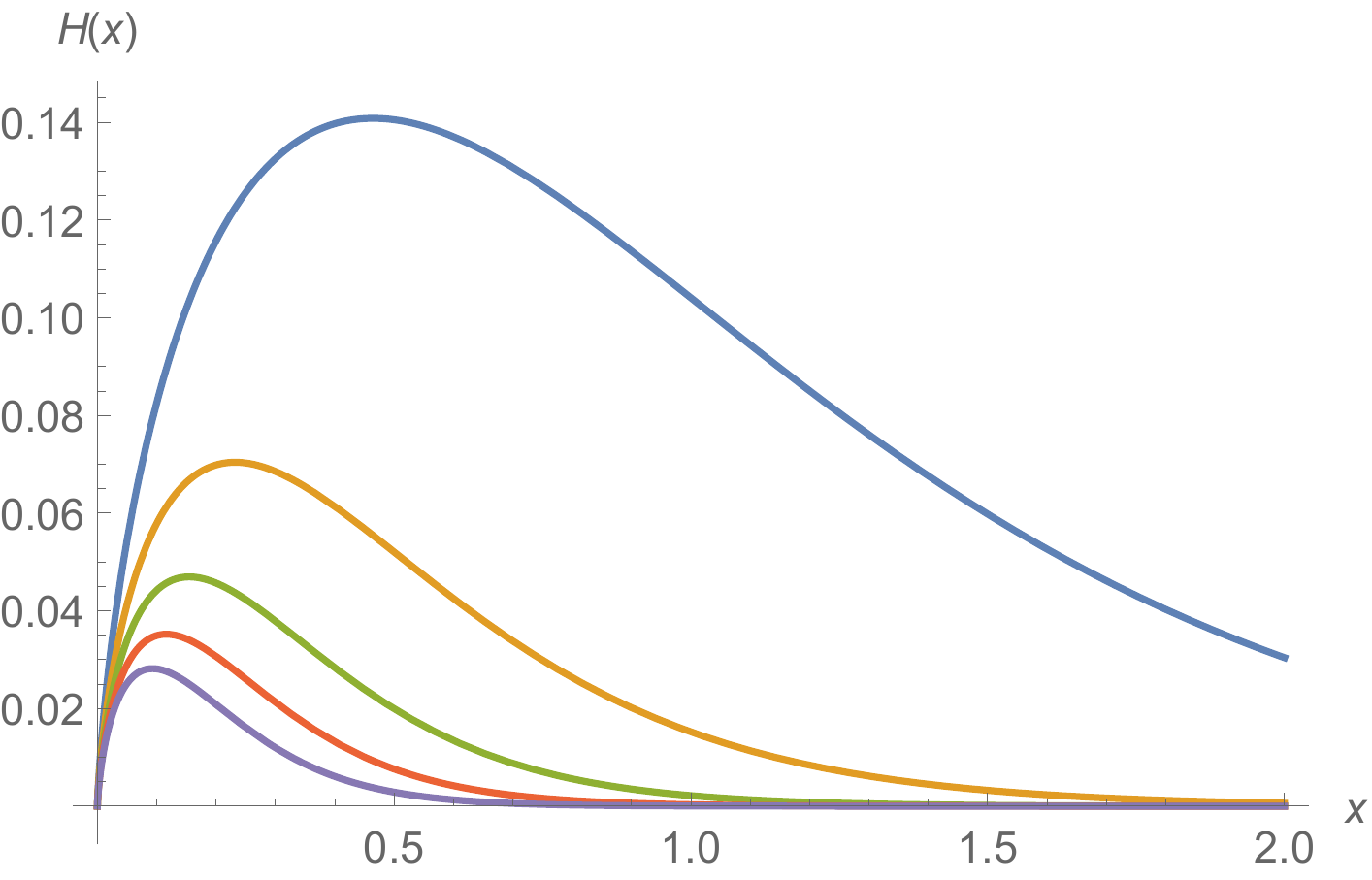}
\caption{Graph of the maximal asymptotic yield $H(x)$ as a function of the harvesting threshold $x$ for different values of the intra-competition rate. We take $\sigma^2=1$, $M=\bar\mu=1$, and $\kappa=1$ (blue), $\kappa=2$ (orange), $\kappa=3$ (green), $\kappa=4$ (red), and $\kappa=5$ (purple).}
\label{fig:varykappa}

\end{center}
\end{figure}
%%%%%%%%%%%%%%%%%%%%

Above we have imposed a bound on the extraction rate, $M$. This was because it is a realistic feature, but it was also practical for the analysis. Nevertheless, it is interesting to consider the case when the extraction rate is unbounded. A practical model with no extraction limit corresponds to having unlimited control over a target population, which is sometimes the case. Such a model would have the benefit of not requiring a nuisance parameter that may be hard to determine.

\subsection{Concave and convex yield functions}
We have also studied the more general case involving concave and convex yield functions. When the yield function is strictly concave, it was shown that the optimal control is not bang-bang, but continuous in the population parameter. Vice-versa, when the yield function is weakly convex, we have shown that the optimal control is necessarily bang-bang. Moreover, if a certain further assumption on the relative rate of growth of $\mu$ and $\Phi$ holds, we can also show that the bang-bang optimal control has a single threshold $x^*$ where the extraction rate goes from 0 to $M$ -- as in the linear special case.
\newline
\indent This generalization allows us to think of applications of population harvesting where the yield function is in fact a utility function, or some other more general social welfare measure.

\subsection{Unbounded harvesting} If we allow for general, possibly unbounded, harvesting we would have to study the Skorokhod SDE
\begin{equation}\label{e:c-harvest_unbounded}
d\tilde X(t) = \tilde X(t)(\mu - \kappa \tilde X(t))\,dt + \sigma \tilde X(t)\,dB(t) - dZ_t, ~\tilde X(0)=x>0.
\end{equation}
where $(Z_t)_{t\geq 0}$ is supposed to be non-negative, increasing, right-continuous and adapted to $(\F_t)_{t\geq 0}$ - we denote the set of all such strategies by $A$. Then the problem is to maximize the asymptotic yield, i.e. find
\[
V(x) = \sup_{(Z_t)_{t\geq 0}\in A}\liminf_{T\to\infty}\E_x \frac{1}{T}\int_0^T dZ_t = \sup_{(Z_t)_{t\geq 0}\in A}\liminf_{T\to\infty} \frac{\E_x Z_T}{T}
\]
We want to find the harvesting strategy $(Z_t^*)_{t\geq 0}\in A$, which we call the optimal harvesting strategy, such that
\[
V(x) =  \liminf_{T\to\infty}\frac{\E_x Z^*_T}{T}.
\]

The analysis above, for the bounded harvesting rate, determined that the optimal strategy has a bang-bang property, where extraction is maximal after some cut-off. This suggests that raising the maximum would not change the bang-bang property, but determining that result required a bounded extraction rate. Thinking of the limiting behavior of the yield function above shows the difficulty: as $M \rightarrow \infty$, the density of the distribution above the cut-off $x^*$ goes to 0 (see \eqref{e:rho}). The conjectured optimal solution is akin to having a reflective boundary at $x^*$, and the yield is determined by the time spent close to the boundary.

\begin{conj}\label{conj:1}
Assume that the population survives in the absence of harvesting i.e. $\mu-\frac{\sigma^2}{2}>0$. The optimal extraction strategy $(Z^*_t)_{t\geq 0}$ has the  form
\begin{equation}\label{e:Z^*}
\begin{aligned}
Z_t^*(x)
&= \begin{cases}
(x-x^*)^+ & \mbox{if $t=0$} \\
L(t,x^*)& \mbox{if $t>0$}.
\end{cases}
\end{aligned}
\end{equation}
for some $x^*\in (0,\infty)$, where $L(t,x^*)$ is the local time at $x^*$ of the process $\tilde X$ from \eqref{e:c-harvest_unbounded}.
\end{conj}
This conjecture is supported by the results from \cite{AS98} where the authors study the maximization of the discounted yield
\[
V(x):=\sup_{(Z_t)_{t\geq 0}\in A}\E_x \int_0^{\tau}e^{-rt} dZ_t
\]
and $\tau:=\inf\{t\geq 0: \tilde X_t=0\}$ is the extinction time. It is shown in \cite{AS98} that the optimal harvesting strategy is of the form \eqref{e:Z^*}. One possible approach to prove Conjecture \ref{conj:1} would be to use the results from \cite{AS98} and then let the discount factor $r$ go to $0$.

{\bf Acknowledgements.}  We thank two anonymous referees for very insightful comments and suggestions that led to major improvements.

\bibliographystyle{amsalpha}
\bibliography{harvest}

\appendix

\section{Proofs}\label{a:proofs}

In this appendix we present the framework of ergodic optimal control and prove the main results of our paper.

For any $v\in\UM$,
denote the unique invariant probability measure of $X(t)$
on $\R_{++}$ by $\pi_v$ if it exists.
Define
$$
\rho_v=\begin{cases}
\int_0^\infty \Phi(xv(x))\pi_v(dx)&\text{ if } \pi_v \text{ exists,}\\
0& \text{ otherwise.}
\end{cases}
$$
Let $p>0$.
Since $\lim_{x\to\infty}\mu(x)=-\infty$,
 there exist constants $ k_{1p},  k_{2p}>0$ such that
\begin{equation}\label{e.3.7.2}
\op_u x^p\leq px^p\mu(x)+\frac{1}{2}p(p-1)\sigma x^p
\leq k_{1p}-k_{2p}x^{p},\, x\in\R_{++}, u\in[0,M]
\end{equation}
By Dynkin's formula
$$\E_x^v [X(t)]^p\leq x^p+k_{1p}t-k_{2p}\E_x^v\int_0^t[X(s)]^{p}ds.$$
Thus,
\begin{equation}\label{e.xp}
\dfrac1t\E_x^v\int_0^t[X(s)]^{p}ds
\leq \frac{1}{k_{2p}}\left(
 \dfrac{x^p}t+k_{1p}\right).
\end{equation}
As a result, the family of occupation measures
$$\Pi_{x,t}^v(\cdot):=\dfrac1t\int_0^t\PP_x^v\{X(s)\in\cdot\}\,ds,~~ t\geq 1$$
is tight.
If $X(t)$ has an invariant probability measure on $\R_{++}$,
then $\left(\Pi_{x,t}^v\right)_{t\geq 0}$ converges weakly to $\pi_v$ because the diffusion is nondegenerate.
This convergence and the uniform integrability \eqref{e.xp}
imply that
$$\lim_{t\to\infty}\dfrac1t\int_0^t\Phi(X(s)v(X(s)))ds=\rho_v.$$
If $X(t)$ has no  invariant probability measures on $\R_{++}$,
then the Dirac measure with mass at $0$ is the only invariant probability measure of $X(t)$
on $\R_+$.
Moreover, any weak-limit of $\left(\Pi_{x,t}^v\right)_{t\geq 0}$ as $t\to\infty$ is an invariant probability measure of $X(t)$
(\cite[Theorem 9.9]{EK09} or \cite[Proposition 8.4]{EHS15}).
Thus, $\left(\Pi_{x,t}^v\right)_{t\geq 0}$ converges weakly to the Dirac measure $\delta_0$ as $t\to\infty$.
Because of \eqref{e.xp} and $\Phi(0)=0$,
we have
$$\lim_{t\to\infty}\dfrac1t\int_0^t\Phi(X(s)v(X(s))ds=\int_0^\infty \Phi(xv(x))\pi_v(dx).$$
Thus,
we always have
\begin{equation}\label{e.lim-rho-v}
\lim_{t\to\infty}\dfrac1t\int_0^t\Phi(X(s)v(X(s))ds=\rho_v.
\end{equation}
Define
\begin{equation}\label{e:rho*}
\rho^*:=\sup_{v\in\UM}\{\rho_v\}.
\end{equation}
It will be shown later that $\rho^*>0$ whenever the population without harvesting persists, i.e. when $\mu(0)-\sigma^2/2>0$.
\begin{thm}
Suppose $\mu(0)-\sigma^2/2>0$, $\mu(\cdot)$ satisfies Assumption \ref{A:1} and $\Phi(\cdot)$ satisfies Assumption \ref{A:2}.
There exists a stationary Markov strategy $v^*\in\UM$ such that
$\pi_{v^*}$ exists and
$\rho_{v^*}=\rho^*$. Moreover,
for any admissible control $h(t)$, we have
$$
\liminf_{T\to\infty}\dfrac1T\int_0^T \Phi\Big(X(t)h(t)\Big)\,dt
\leq \rho_{v^*}=\rho^* \text{ a.s.}
$$
\end{thm}
\begin{proof}
By \eqref{e.xp} and since $\Phi$ has a subpolynomial growth rate we can conclude that
\begin{equation}\label{sup-phi-v}
\sup_{v\in\UM}\int_0^\infty \Phi(xv(x))\pi_v(dx)<\infty.
\end{equation}
Moreover,
since $\mu(0)-\sigma^2/2>0$ we note that, since our population does not go extinct,
$\rho^*>0$.
On the other hand, since $\Phi$ is continuous and $\Phi(0)=0$ we get that $\Phi(x)<\rho^*$ for $x$ is sufficiently small. This fact combined with \eqref{sup-phi-v} implies the
existence of an optimal Markov strategy $v^*$
according to \cite[Theorem 3.4.5, Theorem 3.4.7]{ABG}.
\end{proof}

\begin{thm}\label{thm2.1d}
Suppose $\mu(0)-\sigma^2/2>0$, $\mu(\cdot)$ satisfies Assumption \ref{A:1} and $\Phi(\cdot)$ satisfies Assumption \ref{A:2}. The HJB equation

\begin{equation}\label{e:HJB}
\max_{u\in U}\Big[\op_u V(x)+\Phi(xu)\Big]=\rho
\end{equation}
admits a classical solution $V^*\in C^2(\R_+)$ satisfying $V^*(1)=0$
and $\rho=\rho^*> 0$.
The solution $V^*$ of \eqref{e:HJB} has the following properties:
\begin{itemize}
\item[a)] For any $p\in (0,1)$
\begin{equation}\label{e.growth}
\lim_{x\to\infty}\dfrac{V^*(x)}{x^{p}}=0.
\end{equation}
\item[b)] The function $V^*$ is increasing, that is
\begin{equation}\label{increase}
V^*_x\geq 0,\, x\in\R_{++}.
\end{equation}
\end{itemize}
A Markov control $v$ is optimal if and only if
it satisfies
\begin{equation}\label{e:max_princ}
\begin{aligned}
\dfrac{d V^*}{d x}(x)\Big[x(\mu(x)-v(x))\Big]+\Phi(xv(x))
=\max_{u\in U} \left(\dfrac{d V^*}{d x}(x)\Big[x(\mu(x)-u)\Big]+\Phi(xu)\right)
\end{aligned}
\end{equation}
almost everywhere in $\R_+$.
\end{thm}

\begin{proof}
Consider the optimal problem with the yield function
$$J_h(x)=\E_x\int_0^\infty e^{-\alpha t} h(t)X(t)dt$$
for some fixed $x\in\R_{++}$ and $h\in\mathfrak{U}$. Note that this is the \textit{$\alpha$-discounted} optimal problem.
Pick any $0<x_1<x_2<\infty$ and let $X^{x_1}$, $X^{x_2}$ be the solutions
to the controlled diffusion
$$dX(t) = X(t)(\mu(X(t)) -h(t) )\,dt + \sigma X(t)\,dB(t) $$
with initial values $x_1$, $x_2$ respectively. Note that we are using a fixed admissible control $h(t)$ which is the same for any initial value.
The control $h(t)$ here is not a Markov control which in general depends on the initial value.
Since $\mu(\cdot)$ is continuous and decreasing,
for $y_1, y_2>0$, there exists $\xi(y_1, y_2)>0$ depending continuously on $y_1, y_2$ such that
$\mu(y_1)-\mu(y_2)=-\xi(y_1, y_2)(\ln y_1-\ln y_2)$.
Using It\^{o}'s Lemma
we have
$$
\begin{aligned}
d (\ln X^{x_2}(t)-\ln X^{x_1}(t))=&\left(\mu((X^{x_2}(t))-\mu(X^{x_1}(t))\right)dt\\
=&-\xi(X^{x_1}(t), X^{x_2}(t)) (\ln X^{x_2}(t)-\ln X^{x_1}(t))dt,
\end{aligned}
$$
which in turn yields
 $$
 \ln X^{x_2}(t)-\ln X^{x_1}(t)=(\ln x_2-\ln x_1)\exp\left(-\int_0^t\xi(X^{x_1}(s), X^{x_2}(s))ds\right)
>0.$$
Therefore, if $x_2>x_1$, we get that $$\PP(X^{x_2}(t)> X^{x_1}(t), t\geq0)=1.$$
This implies that
$J_h(\cdot)$ is an increasing function. Therefore, the optimal yield
$$V_\alpha(x):=\sup_{h\in\mathfrak{U}}J_h(x)$$
is also increasing.
By \cite[Lemma 3.7.8]{ABG},
there is a function $V^*\in C^2(\R_{++})$ satisfying
\eqref{e:HJB} for a number $\rho$ such that
\begin{equation}\label{e:rho2}
\rho\geq\rho^*.
\end{equation}
Moreover,
$$V^*(x)=\lim_{n\to\infty} \left(V_{\alpha_n}(x)-V_{\alpha_n}(1)\right)$$
for some sequence $(\alpha_n)_{n\in\N}$ that satisfies $\alpha_n\to0$ as $n\to\infty$. This implies that $V^*$ is an increasing function, i.e.
\begin{equation}\label{increase}
V^*_x\geq 0,\, x\in\R_{++}.
\end{equation}
For any continuous function $\psi:\R_{++}\mapsto\R$
satisfying
\begin{equation}\label{e.op}
|\psi(x)|\leq c(1+x^p), x\in\R_{++}, c>0
\end{equation}
 we have from \eqref{e.3.7.2} and \cite[Lemma 3.7.2]{ABG} that $\E^v_x |\psi(X(t))|$ exists and satisfies
\begin{equation}\label{e.3.7.5}
\lim_{t\to\infty}\left(\frac1t\sup_{v\in\UM}\E^v_x |\psi(X(t))|\right)=0,
\end{equation}
and
\begin{equation}\label{e.3.7.6}
\lim_{R\to\infty}\E^v_x \psi(X(t\wedge\xi_R))=\E^v_x \psi(X(t))<\infty, t\geq0,
\end{equation}
where $\xi_R=\inf\{t\geq0: X(t)> R \text{ or } X(t)<R^{-1}\}.$
Moreover, by using \cite[Lemma 3.7.2]{ABG} again we get that
\begin{equation}\label{e.3.7.4}
\lim_{x\to\infty}\dfrac{f_R(x)}{x^p}=0, R\geq 0
\end{equation}
where
$$f_R(x):=\sup_{v\in\UM}\E^v_x\int_0^{\tau_R}\Phi(X(t))dt,
$$
and
$\tau_R:=\inf\{t\geq0: X(t)\leq R\}.$

By \cite[Formula 3.7.48]{ABG},
we have the estimate
$$
V^*(x)\leq \sup_{v\in\UM}\E^v_x\int_0^{\tau_R}\left(\Phi(X(t))+\rho^*\right)dt+\sup_{y\in[0,R]}\{V^*(y)\}
$$
which implies
\begin{equation}\label{e.V*growth}
V^*(x)\leq c_p(1+x^p), x\geq R\,\text{  for some }\,c_p>0.
\end{equation}

Now, pick any $\varepsilon>0$ and divide \eqref{e.V*growth} on both sides by $x^{p+\varepsilon}$. We get
\[
\frac{V^*(x)}{x^{p+\varepsilon}}\leq c_p\left(\frac{1}{x^{p+\varepsilon}}+x^{-\varepsilon}\right), ~ x\geq R
\]
and by letting $x\to\infty$
\[
\lim_{x\to\infty}\frac{V^*(x)}{x^{p+\varepsilon}} = 0.
\]
This implies, since $p$ and $\varepsilon>0$ are arbitrary, equation \eqref{e.growth}.
Let $\chi:\R_{++}\mapsto[0,1]$ be a continuous function
satisfying $\chi(x)=0$ if $x<\frac12$ and $\chi(x)=1$ if $x\geq 1$.
Then $\psi(x):=V^*(x)\chi(x)$ satisfies \eqref{e.op}  because of \eqref{e.V*growth}. On the other hand, since  $V^*(x)$ is increasing and $V^*(1)=0$, then  $V^*(x)\leq 0$ when $x\leq 1$.
Thus, we have
$$V^*(x)\leq \chi(x)V^*(x), x\in\R_{++}.$$
Let $v^*$ be the measurable function satisfying \eqref{e:max_princ}.
 \begin{equation}\label{e.rho-ine}
\rho\geq\rho^*\geq\rho_{v^*}.
\end{equation}
By Dynkin's formula
$$
\begin{aligned}
\E_x^{v^*}\chi(X(t\wedge\xi_R))V^*(X(t\wedge\xi_R))-V^*(x)&\geq
\E_x^{v^*}V^*(X(t\wedge\xi_R))-V^*(x)\\
&=\E_x^{v^*}\int_0^{t\wedge\xi_R}\left(\rho-\Phi(X(s)v(X(s)))\right)ds
\end{aligned}
$$
Letting $R\to\infty$,
we obtain from the monotone convergence theorem and
\eqref{e.3.7.6} that
$$
\begin{aligned}
\dfrac1t\left(\E_x^{v^*}\chi(X(t))V^*(X(t))-V^*(x)\right)\geq
\rho -\dfrac1t\E_x^{v^*}\int_0^{t}\Phi(X(s)v(X(s)))ds, t>0
\end{aligned}
$$
Letting $t\to\infty$ and using \eqref{e.3.7.5}  and \eqref{e.lim-rho-v}, we have
$$0\geq \rho-\rho_{v^*}.$$
This and \eqref{e.rho-ine} implies that $\rho=\rho^*=\rho_{v^*}$.

By the arguments from \cite[Theorem 3.7.12]{ABG}, we can show that
$v$ is an optimal control if and only if \eqref{e:max_princ} is satisfied.
\end{proof}

When $\Phi$ is the identity mapping the equation \eqref{e:max_princ} becomes
$$
\begin{aligned}
-\dfrac{d V^*}{d x} v (x) + v(x) =\max_{u\in U} \left(-\dfrac{d V^*}{d x} u +u \right),
\end{aligned}
$$
which implies
\begin{equation}\label{e:v_1d}
\begin{aligned}
v(x)
&= \begin{cases}
0 & \mbox{if $ V_x^*>1$} \\
M & \mbox{if $ V_x^*<1$}.
\end{cases}
\end{aligned}
\end{equation}

Our main result is the following theorem.
\main* %% <--- this is defined above; here it's restated...

%\begin{thm}\label{t:main_app}
%Assume that $\Phi(x)=x, x\in(0,\infty)$ and that the population survives in the absence of harvesting, that is $\mu-\frac{\sigma^2}{2}>0$. The optimal control $v$ has the bang-bang form
%\begin{equation*}
%\begin{aligned}
%v(x)
%&= \begin{cases}
%0 & \mbox{if $0< x\leq x^*$} \\
%M & \mbox{if $x>x^*$}
%\end{cases}
%\end{aligned}
%\end{equation*}
%for some $x^*\in (0,\infty)$.
%Furthermore, we have the following upper bound for the optimal yield
%\begin{equation}\label{e:yield_bound}
%\rho^* \leq\frac{\mu^2}{4\kappa}.
%\end{equation}
%\end{thm}
\begin{rmk}\label{r:V_x=1}
If $V_x^*(x)=1$ then we note that \eqref{e:v_1d} does not provide any information about $v(x)$. However, in this case we can set the harvesting rate equal to anything since the yield function will not change. This is because our diffusion is non-degenerate and changing the values of the drift on a set of zero Lebesgue measure does not change the distribution of $X$.
\end{rmk}
We split up the proof of Theorem \ref{t:main} into a few propositions.
It is immediate to see that the HJB equation \eqref{e:HJB} becomes
\begin{equation}\label{e:HJB_1d}
\begin{split}
\rho &=\max_{u\in U}\Big[ x(\mu(x)-u)f_x+\dfrac12\sigma^2 x^2f_{xx} + xu \Big] \\ &=x\mu(x)f_x+\dfrac12\sigma^2x^2f_{xx} +\max_{u\in U}[ (1-f_x)xu]\\
 &= \begin{cases}
x\mu(x)f_x+\dfrac12\sigma^2 x^2f_{xx} & \mbox{if $f_x> 1$}\\
x(\mu(x)-M)f_x+\dfrac12\sigma^2 x^2f_{xx} +Mx &\mbox{if $f_x\leq 1$}.
\end{cases}
\end{split}
\end{equation}

\begin{proof}[Sketch of proof of Theorem \ref{t:main}]
Since the optimal control is given by \eqref{e:v_1d} we need to analyze the properties of the function $V_x^*$ which by \eqref{e:HJB_1d} satisfies a first order ODE. The analysis of this is split up into several propositions. Note that the ODE governing $V_x^*$ is different, depending on whether $V_x^*>1$ or $V_x^*\leq 1$.

In Proposition \ref{p:asymptotic} we analyze the ODE for when $V_x^*\leq 1$ and find its asymptotic behavior close to $0$. Using this we can show in Proposition \ref{p:zero} that one cannot have a $\eta>0$ such that $V_x^*(x)\leq 0$ for all $x\in(0,\eta]$.

Similarly, in Proposition \ref{p:infinity} we show that there can exist no $\zeta>0$ such that $V_x^*(x)\geq 1$ for all $x\geq \zeta$.

In Proposition \ref{p:crossing} we explore the possible ways $V_x^*$ can cross the line $y=1$ and find using soft arguments that there can be at most $3$ crossings. Finally, we show that actually there must be exactly one crossing of $y=1$ by $V_x^*$ and that this crossing has to be from above (Figure \ref{fig:boundarycrossing3}). This combined with \eqref{e:v_1d} completes the proof.
\end{proof}

\begin{prop}\label{p:asymptotic}
Assume $\mu(\cdot)$ is locally Lipschitz on $[0,\infty)$. Then any solution $\varphi_2$ of the ODE
\begin{equation}\label{e:vp2}
\frac{d\vp_2}{dx}(x) + \frac{2(\mu(x) -M)}{\sigma^2 x}\vp_2(x) = \frac{2(\rho-Mx)}{\sigma^2 x^2}
\end{equation}
satisfies
\begin{equation}\label{e:lim_varphi_2=pminfty}
\lim_{x\to 0^+} \varphi_2 (x)  = \pm\infty.
\end{equation}

\end{prop}
\begin{proof}

It follows from the method of integrating factors that the solution to the ODE \eqref{e:vp2} is
\begin{equation}\label{e:Soln_Formula_from_Integrating_Factor}
	\varphi_2 (x) = \frac{\zeta(x_0)\varphi_2 (x_0) + \int_{x_0}^{x} \zeta(y)\beta(y) \;dy}{\zeta(x)},
\end{equation}
where the non-homogeneous term is $\beta(y):=\frac{2(\rho-My)}{\sigma^2 y^2}$, and the integrating factor is
\[
	\zeta(x) := e^{\int_{x_1}^{x} \gamma(y) \;dy},
\]
for $\gamma(y):=\frac{2(\mu(y) -M)}{\sigma^2 y}$, and arbitrary $x_0, x_1\in(0,\infty)$. Since $\mu$ is locally Lipschitz at $x=0$, there are constants $L, K>0$ such that for any $x\in [0,L]$, $|\mu(x)-\mu_0|\leq K x$, where $\mu_0:=\mu(0)$. From now on, we choose $x_1:=L$ (or any number between $0$ and $L$). We have, for any $x\in [0,x_1]$,
\[
	\left|\int_{x_1}^{x} \frac{\mu(y)-\mu_0}{y} \;dy\right| \leq \int_{x}^{x_1} \frac{|\mu(y)-\mu_0|}{y}  \;dy \leq \int_{x}^{x_1} \frac{Ky}{y}  \;dy\leq K (x_1 - x).
\]
This implies that as $x\to 0^+$,
\begin{equation}\label{e:Asym_for_zeta}
	\zeta(x) = e^{\frac{2}{\sigma^2}\int_{x_1}^{x} \frac{\mu(y)-\mu_0}{y} \;dy} e^{\frac{2}{\sigma^2}\int_{x_1}^{x} \frac{\mu_0 -M}{y} \;dy} \sim x^{\frac{2}{\sigma^2}(\mu_0-M)}.
\end{equation}
%Here, we write $A\sim B$ (for any two positive terms $A$ and $B$) if there exist two positive constants $c$ and $C$ such that $cB\leq A\leq CB$.
On the other hand, from now on, if we choose $x_0>0$ sufficiently close to $0$ such that $\rho-Mx>0$ and \eqref{e:Asym_for_zeta} holds for all $x\in (0,x_0)$, then we have, for any $0<x<x_0$,
\begin{equation}\label{e:Asym_for_int_zeta_beta}
\begin{aligned}
	\int_{x}^{x_0} \zeta(y)\beta(y) \;dy &\sim \frac{2}{\sigma^2} \int_{x}^{x_0} y^{\frac{2}{\sigma^2}(\mu_0-M)-2} (\rho-My) \;dy \\
	&=
	\begin{cases}
	C_0 + C_1 x^{\frac{2}{\sigma^2}(\mu_0-M)} + C_2 x^{\frac{2}{\sigma^2}(\mu_0-M)-1} & \text{if } \mu_0-M \neq 0, \; \frac{\sigma^2}{2}\\
	\frac{2}{\sigma^2} (\rho \ln x_0 - M x_0) + \frac{2M}{\sigma^2} x - \frac{2\rho}{\sigma^2} \ln x  & \mbox{if } \mu_0-M = \frac{\sigma^2}{2}\\
 	\frac{2}{\sigma^2} (-\rho x^{-1}_0 - M \ln x_0) + \frac{2M}{\sigma^2} \ln x + \frac{2\rho}{\sigma^2} x^{-1} & \mbox{if } \mu_0-M=0,
    \end{cases}
\end{aligned}
\end{equation}
where the constants $C_i$ are given by
\[
	C_0 := -\frac{M x_0^{\frac{2}{\sigma^2}(\mu_0-M)}}{\mu_0-M} + \frac{\rho x_0^{\frac{2}{\sigma^2} (\mu_0-M)-1}}{\mu_0-\frac{\sigma^2}{2}-M}, \quad C_1 := \frac{M}{\mu_0-M}, \quad C_2 := -\frac{\rho}{\mu_0-\frac{\sigma^2}{2}-M} .
\]

Now, using the asymptotic properties \eqref{e:Asym_for_zeta} and \eqref{e:Asym_for_int_zeta_beta}, we can analyze the limit of $\varphi_2$ as follows.

\noindent\textbf{Case 1:} $\mu_0 < M$.

In this case, we get from \eqref{e:Asym_for_zeta} and \eqref{e:Asym_for_int_zeta_beta} that
\[
	\lim_{x\to 0^+} \zeta (x) = \pm \infty, \quad \lim_{x\to 0^+} \zeta(x_0)\varphi_2 (x_0) + \int_{x_0}^{x} \zeta(y)\beta(y) \;dy = \pm\infty.
\]
Thus, we can apply l'H\^{o}pital's rule and obtain
\begin{equation}\label{e:Apply_l'Hopital_rule_to_Soln_Formula}
	\lim_{x\to 0^+} \varphi_2 (x) = \lim_{x\to 0^+} \frac{\zeta(x_0)\varphi_2 (x_0) + \int_{x_0}^{x} \zeta(y)\beta(y) \;dy}{\zeta(x)} %= \lim_{x\to 0^+} \frac{\beta(x)}{\gamma(x)}
	= \lim_{x\to 0^+} \frac{\rho -Mx}{x(\mu(x)-M)} = \pm\infty
\end{equation}
since $\rho >0$. This shows the limit \eqref{e:lim_varphi_2=pminfty}.

\noindent\textbf{Case 2:} $M\leq\mu_0 \leq M + \frac{\sigma^2}{2}$.
\newline
\indent For this range of $\mu_0$, it follows from \eqref{e:Asym_for_zeta} and \eqref{e:Asym_for_int_zeta_beta} again that
\[
	\lim_{x\to 0^+} \zeta(x_0)\varphi_2 (x_0) + \int_{x_0}^{x} \zeta(y)\beta(y) \;dy = \pm\infty,
\]
but $\lim_{x\to 0^+} \zeta (x)$ exists and is finite. Hence, we can obtain the limit \eqref{e:lim_varphi_2=pminfty} by passing to the limit $x\to 0^+$ in the solution formula~\eqref{e:Soln_Formula_from_Integrating_Factor}.

\noindent\textbf{Case 3:} $\mu_0 > M + \frac{\sigma^2}{2}$.

In this final case, it follows from \eqref{e:Asym_for_zeta} and \eqref{e:Asym_for_int_zeta_beta} that $\lim_{x\to 0^+} \zeta (x) = 0$ and
\[
	J := \lim_{x\to 0^+} \zeta(x_0)\varphi_2 (x_0) + \int_{x_0}^{x} \zeta(y)\beta(y) \;dy
\]
exists and is finite. If $J\neq 0$, then passing to the limit $x\to\infty$ in the solution formula~\eqref{e:Soln_Formula_from_Integrating_Factor} will imply the limit~\eqref{e:lim_varphi_2=pminfty}. Otherwise, we can apply l'H\^{o}pital's rule and do the same computations we did in \eqref{e:Apply_l'Hopital_rule_to_Soln_Formula}. This proves the limit~\eqref{e:lim_varphi_2=pminfty}.

Putting together Cases 1,2 and 3 completes the proof.
\end{proof}

%%%%% Figure %%%
\begin{figure}
\begin{center}
\begin{tikzpicture}[scale=1, line width = 1.2 pt]
\fill[blue!20!white] (2.6,4) .. controls (2.2,2.7) and (1.9,1)..(1.8,0.5) -- (0.5,0.5) -- (0.5,4) -- cycle;
\draw[->] (0.5,0.5) -- (0.5,7) ;
\draw (0.5,7.2) node {$y$};
\draw[->](0,4) -- (8,4);
\draw (0,6) -- (7,6);
\draw (0.2,3.7) node {0};
\draw (0.2,5.7) node {1};
\draw (8,6) node {$y=1$};
\draw (8.2,4) node {$x$};
\draw [name path=search utility, blue!70!black, line width =1.2pt] (3.5,6) .. controls (3.9,6.4) and (4.6,6.6)..(5,6.7);
\draw [name path=search utility, blue!70!black, line width =1.2pt] (3.5,6) .. controls (3.2,5.6) and (3,5.3)..(2.6,4);
\draw [name path=search utility, blue!70!black, line width =1.2pt] (2.6,4) .. controls (2.2,2.7) and (1.9,1)..(1.8,0.5);
\draw (5.8,1.8) node {$\lim_{x\rightarrow 0+}  V^*_x(x)= -\infty$};
\draw[line width = 0.8, dotted] (3.5,4) -- (3.5,6);
\draw (3.5,3.7) node {$x_0$};
\draw (6,6.7) node {$y=V^*_x(x)$};
\draw [line width = 1pt,->](1.5,1) ..  controls (2.5,1) and (3.6,1.3) .. (3.8,1.5);
\end{tikzpicture}

\caption{If $V^*_x$ crosses $y=1$ from below at $x_0$, and it has not crossed from above before then we get a contradiction by Proposition \ref{p:zero}.}
\label{fig:boundarycrossing2}

\end{center}
\end{figure}
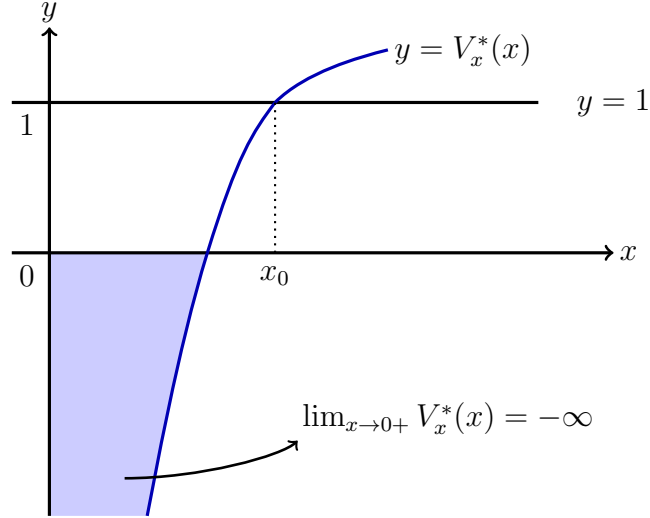
%%%%%%%%%%%%%%%%%%%

\begin{prop}\label{p:zero} There does not exist any $\eta>0$ such that $V_x^*(x)\leq 1, x\in (0,\eta]$.
\end{prop}
\begin{proof}
We will argue by contradiction. Assume there exists $\eta>0$ such that $V_x^*(x)\leq 1, x\in (0,\eta]$.
Then by \eqref{e:HJB_1d} we get that $V_x^*$ follows the ODE \eqref{e:vp2}
for all $x\in (0,\eta)$.
Making use of Proposition \ref{p:asymptotic} we get that
\[
\lim_{x\to 0+}V_x^*(x) = \lim_{x\to 0+}\vp_2(x) = \pm \infty
\]
which contradicts that $V_x^*\geq 0$ or that $V_x^*(x)\leq 1, x\in(0,\eta]$.
The proof is complete.
\end{proof}
The above Proposition shows that the scenario from Figure \ref{fig:boundarycrossing2} cannot happen.

\begin{prop}\label{p:infinity}
There does not exist any $\chi>0$ such that $V_x^*(x)\geq 1$ for all $x\geq \chi$.
\end{prop}
\begin{proof}
Once again we will argue by contradiction. Assume there exists $\chi>0$ such that $V_x^*(x)\geq 1$ for all $x\geq \chi$.
By \eqref{e:HJB_1d} $V_x^*$ will follow the ODE
\[
\frac{d\vp_1}{dx}(x) + \frac{2\mu(x)}{\sigma^2 x}\vp_1(x) = \frac{2\rho}{\sigma^2 x^2}
\]
for all $x\geq \chi$. As a result we get just as in Proposition \ref{p:asymptotic}
\begin{equation}\label{e:Soln_Formula_from_Integrating_Factor2}
	\varphi_1 (x) = \frac{\zeta(x_0)\varphi_1 (x_0) + \int_{x_0}^{x} \zeta(y)\beta(y) \;dy}{\zeta(x)},
\end{equation}
where the non-homogeneous term is $\beta(y):=\frac{2\rho}{\sigma^2 y^2}$, and the integrating factor is
\[
	\zeta(x) := e^{\int_{x_1}^{x} \gamma(y) \;dy},
\]
for $\gamma(y):=\frac{2\mu(y)}{\sigma^2 y}$, and arbitrary $x_0, x_1\in(\chi,\infty)$.
Under Assumption \ref{A:1} we can see that there exist constants $L>0$ and $c>0$ such that $\mu(y) < -c$ for all $y>L$, and hence, $\int_{L}^x \frac{\mu(y)}{y} \;dy \leq -c \int_{x_1}^x \frac{1}{y} \;dy = -c (\ln x - \ln x_1) \to -\infty$ as $x\to\infty$. If we choose $c> \frac{\sigma^2}{2}$, $x_1:=L$ we get
\begin{equation}\label{e:growth}
 x \zeta(x) \leq x^{1 - \frac{2c}{\sigma^2}} \to 0
\end{equation}
 as $x\to\infty$.
If
\[
\zeta(x_0)\varphi_1 (x_0) + \int_{x_0}^{\infty} \zeta(y)\beta(y) \;dy>0
\]
then by \eqref{e:growth} and the positivity of $\zeta$ one has
\[
\lim_{x\to\infty} \frac{V_x^*}{x}= \lim_{x\to\infty} \frac{\vp_1(x)}{x} = \frac{\zeta(x_0)\varphi_1 (x_0) + \int_{x_0}^{x} \zeta(y)\beta(y) \;dy}{x\zeta(x)}= + \infty
\]
which contradicts the growth condition \eqref{e.growth}.
Therefore we need
\[
\zeta(x_0)\varphi_1 (x_0) + \int_{x_0}^{\infty} \zeta(y)\beta(y) \;dy \leq 0.
\]
Note that in this case
\begin{equation*}
\begin{split}
\zeta(x_0)\varphi_1 (x_0) + \int_{x_0}^{x} \zeta(y)\beta(y) \;dy \leq  - \int_{x}^{\infty} \zeta(y)\beta(y) \;dy <0.
\end{split}
\end{equation*}
This implies, since $\zeta(x)>0$, that for $x>x_0$
\begin{equation*}
\begin{split}
V^*_x(x) &= \varphi_1 (x) = \frac{\zeta(x_0)\varphi_1 (x_0) + \int_{x_0}^{x} \zeta(y)\beta(y) \;dy}{\zeta(x)}<0,
\end{split}
\end{equation*}
which contradicts the assumption that $V_x^*(x)\geq 1$ for all $x\geq \chi$.
\end{proof}
The above Proposition shows that the scenario from Figure \ref{fig:boundarycrossing1} is not possible.

%%%%% Figure %%%
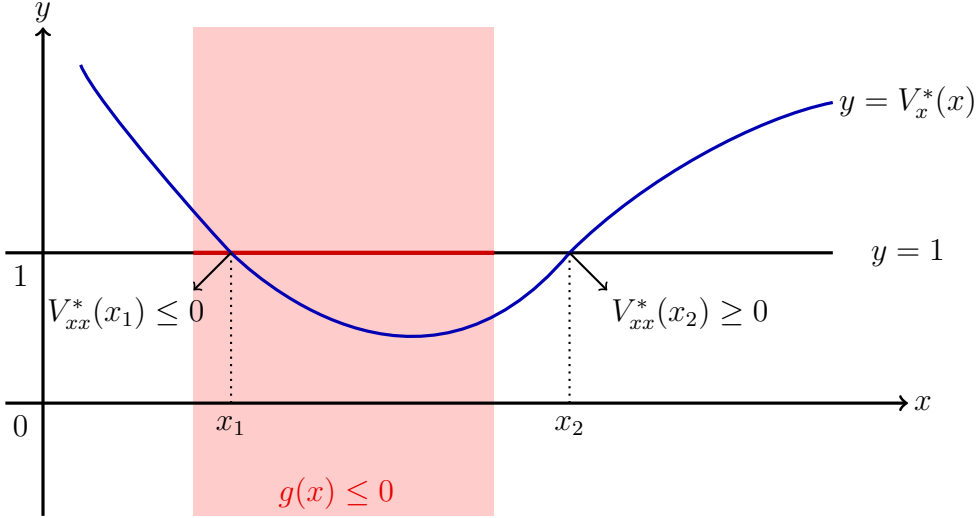
\begin{figure}
\begin{center}
\begin{tikzpicture}[scale=1, line width = 1.2 pt]
\fill[red!20!white] (2.5,0.5) rectangle +(4,6.5);
\draw[->] (0.5,0.5) -- (0.5,7) ;
\draw (0.5,7.2) node {$y$};
\draw[->](0,2) -- (12,2);
\draw (0,4) -- (11,4);
\draw[red!80!black, line width = 1.6 pt] (2.5,4) -- (6.5,4);
\draw (0.2,1.7) node {0};
\draw (0.2,3.7) node {1};
\draw (12,4) node {$y=1$};
\draw (12.2,2) node {$x$};
\draw [name path=search utility, blue!70!black, line width =1.2pt] (7.5,4) .. controls (8.5,5) and (10,5.8)..(11,6);
\draw [name path=search utility, blue!70!black, line width =1.2pt] (7.5,4) .. controls (5.5,1.6) and (3.2,3.8)..(3,4);
\draw [name path=search utility, blue!70!black, line width =1.2pt] (3,4) .. controls (2.8,4.2) and (1.2,6)..(1,6.5);
\draw[red!90!black] (4.4,0.8) node {$g(x) \leq 0$};
\draw[line width = 0.8, dotted] (3,2) -- (3,4);
\draw[line width = 0.8, dotted] (7.5,2) -- (7.5,4);
\draw (3,1.7) node {$x_1$};
\draw (7.5,1.7) node {$x_2$};
\draw (12,6) node {$y=V^*_x(x)$};
\draw[line width = 0.8, ->] (7.5,4) -- (8,3.5);
\draw (9.1,3.2) node {$V^*_{xx}(x_2) \geq 0$};
\draw[line width = 0.8, ->] (3,4) -- (2.5,3.5);
\draw (1.6,3.2) node {$V^*_{xx}(x_1) \leq 0$};
\end{tikzpicture}

\caption{An impossible scenario, by Proposition \ref{p:infinity}.}
\label{fig:boundarycrossing1}

\end{center}
\end{figure}
%%%%%%%%%%%%%%%%%%%

Set $g(x):=\rho - x\mu(x)$. By assumption $p(x):=x\mu(x)$ has a unique maximum and $\mu$ is locally Lipschitz and decreasing with $\lim_{x\to\infty} \mu(x)=-\infty$. This implies that $g(x)$ has a unique minimum for some $x_\iota\in(0,\infty)$ \footnote{$|\mu(x) - \mu(0)| < M|x|$ for some real $M>0$ as $\mu$ locally Lipschitz by asssumption. Therefore $|g(x) - g(0)| < Mx^2$, so $g$ differentiable at 0. Moreover, $g'(0) = - \mu(0)$, and $\mu(0)>0$ by assumption, so $x_\iota \neq 0$.}. If $g(x_\iota)<0$ then $g$ intersects the $x$ axis in exactly two points $0<\alpha_1<\alpha_2<\infty$. If $g(x_\iota)>0$ there is no intersection of $g$ with the $x$ axis. Finally, if $g(x_\iota)=0$ there is exactly one intersection and this happens at $x=x_\iota$.
\begin{prop}\label{p:crossing}
The function $V^*_x$ crosses the line $y=1$ at most three times. More specifically, we have the following possibilities:
\begin{enumerate}[(I)]
\item If  $g(x_\iota)<0$ then
\begin{enumerate}[(i)]
\item For $0\leq x<\alpha_1$ the function $V_x^*$ can only pass the line $y=1$ at most once and the crossing has to be from below.
\item For $x>\alpha_2$ the function $V_x^*$ can pass the line $y=1$ at most once and the crossing has to be from below.

\item For $\alpha_1<x<\alpha_2$ the function $V_x^*$ can pass the line $y=1$ at most once and the crossing has to be from above.
\end{enumerate}
\item If  $g(x_\iota)>0$ then the function $V_x^*$ can pass the line $y=1$ at most once and the crossing has to be from below.

\item If  $g(x_\iota)=0$ then $V_x^*$ can cross the line $y=1$ at most three times. In particular, the possible crossing(s) in $(0,x_\iota) \cup (x_\iota,\infty)$ must be from below.

\end{enumerate}

\end{prop}
\begin{proof}
It follows from the HJB equation \eqref{e:HJB} with $\vp:=V_x$ that if $\varphi(x_0)=1$, then we have
\[
g(x_0) = \rho - x_0\mu (x_0) = \frac{1}{2}\sigma^2 x_0^2 \varphi'(x_0)\begin{cases}
<0 & \mbox{if $\varphi'(x_0)<0$} \\
=0 & \mbox{if $\varphi'(x_0)=0$} \\
>0 & \mbox{if $\varphi'(x_0)>0$}.
\end{cases}
\]
Therefore, when $\varphi$ crosses the line $y=1$, we obtain some information from $g$.
More precisely, we can infer the following:
\begin{enumerate}[(I)]
\item When $g(x_\iota)<0$ the function $g(x) = \rho - x\mu (x)$ has exactly two zeros at $\alpha_1, \alpha_2$ with $0<\alpha_1<\alpha_2<\infty$.
\begin{enumerate}[(i)]
\item for $0\leq x<\alpha_1$ we have $g(x) > 0$, hence $\varphi$ is only allowed to cross the line $y=1$ from below in this region;
\item for  $x>\alpha_2$ we have $g(x) > 0$, hence $\varphi$ is only allowed to cross the line $y=1$ from below in this region;
\item for $\alpha_1<x<\alpha_2$, $g(x)<0$ and $\varphi$ is only allowed to cross the line $y=1$ from above in this region.
\end{enumerate}
\item If $g(x_\iota)>0$ then $g(x)> 0$ for all $x\in\R_+$.
The function $V_x^*$ can pass the line $y=1$ at most once and the crossing has to be from below.
\item If $g(x_\iota)=0$ then  $g(x)$ has a unique intersection of the $x$ axis at $x_\iota$.
As a consequence $g(x)\geq 0$ and the function $V_x^*$ can pass the line $y=1$ at most thrice: at most once from below in the region $x<x_\iota$, at most once from below in the region $x>x_\iota$ and at most once from above or from below at the point  $x= x_\iota$.
\end{enumerate}

\end{proof}
\begin{rmk}
By the analysis above one can note that at the intersection points (or roots) $\alpha_{1,2}$ of the function $g(x)$ with the $x$ axis the derivative of $\varphi$ is $0$. This makes it more complicated to say, in case there is a crossing at a root, if the crossing is from above or from below. However, this does not require us to change our arguments. For example, if there is a crossing from below on $0\leq x<\alpha_1$ and there is a crossing at $x=\alpha_1$ then the crossing at $\alpha_1$ is necessarily from above. This then implies that there can be no crossing for $x\in(\alpha_1,\alpha_2)$ because in this region the crossing has to be from above and there cannot be two crossings from above in a row.
\end{rmk}

\begin{proof}[Proof of Theorem \ref{t:main}]
A direct consequence of Proposition \ref{p:crossing} is that $V_x^*$ can cross the line $y=1$ at most three times. We also know, given the at most two possible solutions $\alpha_{1,2}$ of the equation $g(x)=0$ how these crossings have to happen. Next, we eliminate all but one possibility.
\begin{enumerate}[i)]
\item If we get a crossing from below in $(0,\alpha_1)$ this means that there exists $\eta>0$ such that for all $x\in (0,\eta)$ we have  $V_x^*(x)=\varphi_2(x) \leq 1$. This is not possible by Proposition \ref{p:zero}. As such there can be no crossings in $(0,\alpha_1)$.
\item If we have a crossing from below in $(\alpha_2,\infty)$ then there is $\zeta>0$ such that for all $x\geq \zeta$
\[
V_x^*(x)=\varphi_1(x)\geq 1.
\]
This is not possible by Proposition \ref{p:infinity}.
Therefore, there are no crossings in $ (\alpha_2,\infty)$.
\item We cannot have that $V_x^*(x)\geq 1$ for all $x\in(0,\infty)$ because then we get a contradiction by Proposition \ref{p:infinity}. Similarly, we cannot have $ V_x^*(x)\leq 1$ for all $x\in(0,\infty)$ since we get a contradiction by Proposition \ref{p:zero}.
\item If $g(x_\iota)>0$ then, in principle, there could be at most one crossing and this would have to be from below. But this creates a contradiction by either using Proposition \ref{p:zero} or Proposition \ref{p:infinity}. If there is no crossing then we get a contradiction by (iii) above.
\item If $g(x_\iota)=0$ then
\begin{enumerate}
\item If there is no crossing, then we get a contradiction by part iii) above.
\item If there are two crossings then we get contradictions from either Proposition \ref{p:zero} or Proposition \ref{p:infinity}.
\item If there are three crossings then we must have a crossing from below in $(0,x_\iota)$, one from above at $x=x_\iota$ and one from below in $(x_\iota,\infty)$. This yields a contradiction because of Proposition \ref{p:zero}.
\item If there is just one crossing and the crossing is from below then we get a contradiction by Proposition \ref{p:infinity}.
\end{enumerate}
\item By parts i)-iv) we get that there is exactly one crossing of the line $y=1$, that this crossing is from above and that the crossing happens at a point in the interval $[\alpha_1,\alpha_2]$ when $g(x_\iota)<0$ or at $x_\iota$ if $g(x_\iota)=0$.
\end{enumerate}

%%%%% Figure %%%
\begin{figure}
\begin{center}
\begin{tikzpicture}[scale=1, line width = 1.2 pt]
\fill[red!20!white] (2.5,0.5) rectangle +(4,6.5);
\draw[->] (0.5,0.5) -- (0.5,7) ;
\draw (0.5,7.2) node {$y$};
\draw[->](0,2) -- (12,2);
\draw (0,4) -- (11,4);
\draw[red!80!black, line width = 1.6 pt] (2.5,4) -- (6.5,4);
\draw (0.2,1.7) node {0};
\draw (0.2,3.7) node {1};
\draw (12,4) node {$y=1$};
\draw (12.2,2) node {$x$};
\draw [name path=search utility, blue!70!black, line width =1.2pt] (7.5,2.3) .. controls (8.5,2.2) and (10,2.2)..(11,2.2);
\draw [name path=search utility, blue!70!black, line width =1.2pt] (7.5,2.3) .. controls (5.5,2.4) and (3.6,3.3)..(3,4);
\draw [name path=search utility, blue!70!black, line width =1.2pt] (3,4) .. controls (2.8,4.2) and (1.2,6)..(1,6.5);
\draw[red!90!black] (4.4,0.8) node {$g(x) \leq 0$};
\draw[line width = 0.8, dotted] (3,2) -- (3,4);
%\draw[line width = 0.8, dotted] (7.5,2) -- (7.5,4);
\draw (3,1.7) node {$x_1$};
%\draw (7.5,1.7) node {$x_2$};
\draw (12,2.4) node {$y=V^*_x(x)$};
\end{tikzpicture}

\caption{The only case which doesn't lead to a contradiction is when $V^*_x$ crosses $y=1$ only once and the crossing is from above.}
\label{fig:boundarycrossing3}

\end{center}
\end{figure}
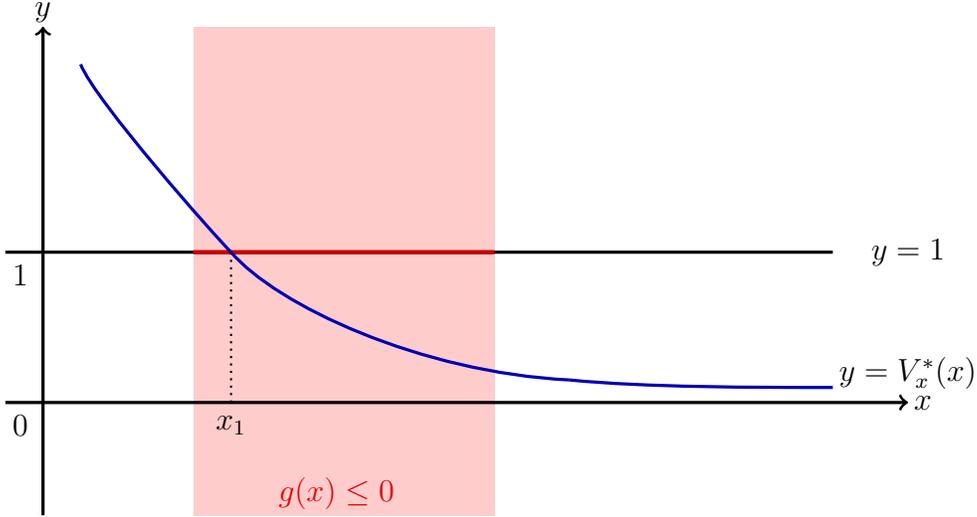

%%%%%%%%%%%%%%%%%

This, together with \eqref{e:v_1d}, implies that the optimal strategy is of bang-bang type
\[
\begin{aligned}
v(x)
&= \begin{cases}
0 & \mbox{if $0< x\leq x^*$} \\
M & \mbox{if $x>x^*$.}
\end{cases}
\end{aligned}
\]
Moreover, one can see that $g(x_\iota)\leq 0$ which in turn forces
\[
\rho \leq \sup_{x\in \R_+} x\mu(x) = x_\iota \mu(x_\iota).
\]

\end{proof}

\section{Optimal harvesting with concave and convex yields: proofs}\label{s:cont}

This appendix shows that for a class of yield functions $\Phi$ one can get \textit{continuous} optimal harvesting strategies.
Therefore, the optimal harvesting strategy will be discontinuous. One might wonder under which conditions on $\Phi$ the optimal harvesting strategies will be continuous.

We proved in Theorem \ref{thm2.1d} that the HJB equation

\begin{equation*}
\max_{u\in U}\Big[\op_u V(x)+\Phi(xu)\Big]=\rho
\end{equation*}
admits a classical solution $V^*\in C^2(\R_+)$ satisfying $V^*(1)=0$ and $\rho=\rho^*> 0$.

For any given $\Phi$, we define
\[
F(\omega) := -A\omega + \Phi (\omega),
\]
where $A$ is a shorthand of $V_x^*$, that is,
\[
A := V_x^*(x).
\]
For any fixed $x$, we can see $A$ as a constant. Using these shorthands, we can rewrite the HJB equation as
\begin{equation}\label{e:Simplified_Min_Prin_in_F}
\begin{aligned}
F(xv) = \max_{\omega\in [0,L]} F(\omega),
\end{aligned}
\end{equation}
where $L:=xM$. A direct computation yields
\[\left\{\begin{aligned}
F(0) &= -A\cdot 0 + \Phi(0) = 0 \\
F(L) &= -AL + \Phi(L)\\
F'(\omega) &= -A + \Phi'(\omega)
\end{aligned}\right.\]
because $\Phi (0) = 0$. Therefore, the critical point(s) will be given by $\omega_c = [\Phi']^{-1}(A)$, and
\[
F(\omega_c) = -A\omega_c + \Phi(\omega_c) = -A[\Phi']^{-1}(A) + \Phi([\Phi']^{-1}(A)).
\]
If $\Phi$ is assumed to be strictly concave, the maximum on the right hand side of \eqref{e:Simplified_Min_Prin_in_F} can be found easily because $F''=\Phi''$.
\cont*

\begin{proof}
Assume that $\Phi$ is $C^2$ and strictly concave.
Since $\Phi$ is $C^2$ we have that $\Phi''<0$. In this case, $\Phi'$ is strictly decreasing, so its inverse is well-defined. As a result, we have a unique critical point which is a maximum $\omega_c = [\Phi']^{-1}(A)$. A standard calculus result yields
\[\begin{aligned}
\max_{\omega\in [0,L]} F(\omega) &= \begin{cases}
\max \left\{ F(0),F(L) \right\} & \mbox{if $\omega_c \not\in (0,L)$} \\
\max \left\{ F(0),F(\omega_c),F(L) \right\} & \mbox{if $0<\omega_c<L$}
\end{cases}\\
&= \begin{cases}
0 & \mbox{if $\omega_c \leq 0$} \\
F(\omega_c) & \mbox{if $0<\omega_c<L$}\\
F(L) & \mbox{if $\omega_c \geq L$},
\end{cases}
\end{aligned}\]
where we used the fact that $F(0)=0$ and the concavity of $\Phi$ in the last equality.

Depending on the maximum point, we have the corresponding optimal Markov control:
\[\begin{aligned}
v &= \begin{cases}
0 &\mbox{if $\displaystyle\max_{\omega\in [0,L]} F(\omega)=0$}\\
\displaystyle\frac{[\Phi']^{-1}(A)}{x} &\mbox{if $\displaystyle\max_{\omega\in [0,L]} F(\omega)=F(\omega_c)$}\\
M &\mbox{if $\displaystyle\max_{\omega\in [0,L]} F(\omega)=F(L)$}
\end{cases} \\
&= \begin{cases}
0 &\mbox{if $[\Phi']^{-1}(A)\leq 0$}\\
\displaystyle\frac{[\Phi']^{-1}(A)}{x} &\mbox{if $0<[\Phi']^{-1}(A)<xM$}\\
M &\mbox{if $[\Phi']^{-1}(A)\geq xM$},
\end{cases} \\
\end{aligned}\]
because $v$ is the solution to
\[
-Axv + \Phi(xv) = \max_{\omega\in [0,L]} F(\omega).
\]

In conclusion, in this case, $v$ depends on $A:=\dfrac{d V^*}{d x}(x)$ continuously. Hence, since $V^*\in C^2\left(\R_+\right)$ we conclude that $v$ is continuous.

The HJB equation \eqref{e:HJB} becomes
\begin{equation*}
\begin{split}
\rho &=\max_{u\in U}\Big[ x(\mu(x)-u)f_x+\dfrac12\sigma^2 x^2f_{xx} + \Phi(xu) \Big] \\ &=x\mu(x)f_x+\dfrac12\sigma^2x^2f_{xx} +\max_{u\in U}[ (\Phi(xu)-xu f_x)]\\
 &= \begin{cases}
x\mu(x)f_x+\dfrac12\sigma^2 x^2f_{xx} & \mbox{if $[\Phi']^{-1}(f_x(x))\leq 0$},\\
x\mu(x)f_x+\dfrac12\sigma^2 x^2f_{xx} -f_x[\Phi']^{-1}(f_x) + \Phi([\Phi']^{-1}(f_x)) &\mbox{if $0<[\Phi']^{-1}(f_x(x))<xM$},\\
x(\mu(x)-M)f_x+\dfrac12\sigma^2 x^2f_{xx}+\Phi(xM) &\mbox{if $[\Phi']^{-1}(f_x(x))\geq xM$.}
\end{cases}
\end{split}
\end{equation*}

\end{proof}

The case when the yield function $\Phi$ is convex is qualitatively similar to the case when the yield function is linear, and the optimal solution is of the bang-bang type. We can improve Theorem \ref{t:main} as follows.
\mainvex*

\begin{proof}
This proof is similar to the proof of Theorem \ref{t:main} in the appendix. By \eqref{e:max_princ}
\[
\frac{d V^*}{dx}(x)[x(\mu(x)-v(x))] + \Phi(xv(x)) = \max_{u \in U}\left( \frac{d V^*}{dx}(x)[x(\mu(x)-u)] + \Phi(xu)\right).
\]
Dropping the common terms gives
\[
- \frac{d V^*}{dx}(x) xv(x) + \Phi(xv(x)) = \max_{u \in U}\left( - \frac{d V^*}{dx}(x)xu + \Phi(xu)\right).
\]
With $x>0$, the right hand side is a weakly convex function of $u$, so one of the end points of the interval $U$ achieves the maximum. This already shows that the optimal control is bang-bang, but says nothing else of the shape of $v(x)$.
\newline
$\Phi(0) = 0$, so
\[
\max_{u \in U}\left( -\frac{d V^*}{dx}(x)xu + \Phi(xu)\right) =
\begin{cases}
-\frac{d V^*}{dx}(x)xM + \Phi(xM), & \text{ if }\frac{\Phi(xM)}{xM} > \frac{d V^*}{dx}(x), \\
0, & \text{ else.}
\end{cases}
\]
This implies
\[
v(x) =
\begin{cases}
0, & \text{ if } V^*_x < \frac{\Phi(xM)}{xM}, \\
M, & \text{ if } V^*_x \geq \frac{\Phi(xM)}{xM}.
\end{cases}
\]
$\Phi(\cdot)$ is weakly convex, therefore, for $\alpha \in (0,1)$, $\Phi(\alpha x + (1-\alpha) y ) \leq \alpha \Phi(x) + (1-\alpha)\Phi(y)$. By assumption, it is also continuous and positive valued. So, for $\alpha \in (0,1)$, $\alpha \Phi(xM) \geq \Phi(\alpha xM)$, equivalent with $\Phi(xM) \geq \frac{1}{\alpha} \Phi(\alpha xM)$, equivalent with $\frac{\Phi(xM)}{xM} \geq \frac{\Phi(\alpha xM)}{\alpha xM}$ if $x,M>0$. Therefore $\frac{\Phi(xM)}{xM}$ must be positive and monotonically increasing in $x$ for $M>0$, $x>0$. In particular $\Phi'(0) = \lim_{x\rightarrow 0^+} \frac{\Phi(xM)}{xM}$ exists and it is greater or equal to 0.
\newline
The HJB equation \ref{e:HJB} becomes
\begin{equation}\label{e:HJB2}
\rho =
\begin{cases}
x\mu(x)f_x+\dfrac12\sigma^2 x^2f_{xx} & \mbox{if $f_x> \frac{\Phi(xM)}{xM}$},\\
x(\mu(x)-M)f_x+\dfrac12\sigma^2 x^2f_{xx} +\Phi(Mx) &\mbox{if $f_x\leq \frac{\Phi(xM)}{xM}$}.
\end{cases}
\end{equation}
One can easily modify the proofs from Appendix \ref{a:proofs} to show the following four propositions:
\begin{prop}\label{p:asymptotic2}
Assume $\mu, \Phi$ satisfy the assumptions of Theorem \ref{t:mainvex}. Then any solution $\varphi_2$ of the ODE
\begin{equation}\label{e:vp22}
\frac{d\vp_2}{dx}(x) + \frac{2(\mu(x) -M)}{\sigma^2 x}\vp_2(x) = \frac{2(\rho-\Phi(Mx))}{\sigma^2 x^2}\end{equation}
satisfies
\begin{equation}\label{e:lim_varphi_2=pminfty}
\lim_{x\to 0^+} \varphi_2 (x)  = \pm\infty.
\end{equation}
\end{prop}
\begin{proof}
Proceed similarly to the proof of Proposition \ref{p:asymptotic}, replacing the definition $\beta(y): = \frac{2(\rho - \Phi(My))}{\sigma^2 y^2}$. This time,
\begin{equation*}
\int_{x}^{x_0} \zeta(y)\beta(y) \;dy \sim \frac{2}{\sigma^2} \int_{x}^{x_0} y^{\frac{2}{\sigma^2}(\mu_0-M)-2} (\rho-\Phi(My)) \;dy.
\end{equation*}
For $y \in [0,x_0]$, we have $\Phi'(0) \leq \frac{\Phi(My)}{My} \leq \frac{\Phi(Mx_0)}{Mx_0}$, so
\begin{align*}
\frac{2}{\sigma^2} \int_{x}^{x_0} y^{\frac{2}{\sigma^2}(\mu_0-M)-2} (\rho-\frac{\Phi(Mx_0)}{Mx_0}My) \;dy &\leq
\frac{2}{\sigma^2} \int_{x}^{x_0} y^{\frac{2}{\sigma^2}(\mu_0-M)-2} (\rho-\Phi(My)) \;dy \\
 &\leq \frac{2}{\sigma^2} \int_{x}^{x_0} y^{\frac{2}{\sigma^2}(\mu_0-M)-2} (\rho-\Phi'(0) My) \;dy.
\end{align*}
For a general positive constant $N$,
\[
\frac{2}{\sigma^2} \int_{x}^{x_0} y^{\frac{2}{\sigma^2}(\mu_0-M)-2} (\rho-Ny) \;dy =
	\begin{cases}
	C_0 + C_1 x^{\frac{2}{\sigma^2}(\mu_0-M)} + C_2 x^{\frac{2}{\sigma^2}(\mu_0-M)-1} & \text{if } \mu_0-M \neq 0, \; \frac{\sigma^2}{2}\\
	\frac{2}{\sigma^2} (\rho \ln x_0 - N x_0) + \frac{2N}{\sigma^2} x - \frac{2\rho}{\sigma^2} \ln x  & \mbox{if } \mu_0-M = \frac{\sigma^2}{2}\\
 	\frac{2}{\sigma^2} (-\rho x^{-1}_0 - N \ln x_0) + \frac{2N}{\sigma^2} \ln x + \frac{2\rho}{\sigma^2} x^{-1} & \mbox{if } \mu_0-M=0,
    \end{cases}
\]
where the integration constants are given by
\[
	C_0 := -\frac{N x_0^{\frac{2}{\sigma^2}(\mu_0-M)}}{\mu_0-M} + \frac{\rho x_0^{\frac{2}{\sigma^2} (\mu_0-M)-1}}{\mu_0-\frac{\sigma^2}{2}-M}, \quad C_1 := \frac{N}{\mu_0-M}, \quad C_2 := -\frac{\rho}{\mu_0-\frac{\sigma^2}{2}-M} .
\]
Now the case-by-case analysis of Proposition \ref{p:asymptotic} can be repeated similarly because the constants of the dominant terms in the expression above do not depend on $N$.
\end{proof}
\begin{prop}\label{p:zero2} There does not exist any $\eta>0$ such that $V_x^*(x)\leq \frac{\Phi(xM)}{xM}, x\in (0,\eta]$.
\end{prop}
\begin{proof}
Noting that $\sup_{x\in(0,\eta]} \frac{\Phi(xM)}{xM} = \frac{\Phi(\eta M)}{\eta M}$, the proof is similar to the proof of Proposition \ref{p:zero}, relying on the application of Proposition \ref{p:asymptotic2} to equation (\ref{e:vp22}).
\end{proof}
\begin{prop}\label{p:infinity2}
There does not exist any $\chi>0$ such that $V_x^*(x)\geq \frac{\Phi(xM)}{xM}$ for all $x\geq \chi$.
\end{prop}
\begin{proof}
It follows the proof of Proposition \ref{p:infinity} without change, because $\frac{\Phi(xM)}{xM} \geq 0$.
\end{proof}
%%%%% Figure %%%
\begin{figure}
\begin{center}
\begin{tikzpicture}[scale=1, line width = 1.2 pt]
\fill[red!20!white] (2.5,-0.2) rectangle +(4,6.5);
\draw[->] (0.5,0) -- (0.5,7) ;
\draw (0.5,7.2) node {$y$};
\draw[->](0,2) -- (12,2);
\draw (0.2,1.7) node {0};
\draw (0,2.4) node {$\Phi'(0)$};
\draw [name path=phi, line width =1.2pt] (0.5,2.4) .. controls (1.5,2.6) and (1.5,2.6)..(2.5,3.2);
\draw [name path=phi, color = red!80!black, line width =1.2pt] (2.5,3.2) .. controls (3.5,3.8) and (4.7,3.7)..(6.5,4.2);
\draw [name path=phi2, line width =1.2pt] (6.5,4.2) .. controls (8.3,4.7) and (9,5.5) .. (11,6);
\draw (12,6) node {$\frac{\Phi(xM)}{xM}$};
\draw (12.2,2) node {$x$};
\draw [name path=search utility, blue!70!black, line width =1.2pt] (7.5,2.3) .. controls (8.5,2.2) and (10,2.2)..(11,2.2);
\draw [name path=search utility, blue!70!black, line width =1.2pt] (7.5,2.3) .. controls (5.5,2.4) and (4.6,3.3)..(4,4);
\draw [name path=search utility, blue!70!black, line width =1.2pt] (2,3.5) .. controls (2.5,3) and (3.4,4.7)..(4,4);
\draw [name path=search utility, blue!70!black, line width =1.2pt] (2,3.5) .. controls (1.8,3.7) and (1.2,4)..(1,6.5);
\draw[red!90!black] (4.4,0.8) node {$V^*_{xx} \leq \left(\frac{\Phi(xM)}{xM}\right)'$};
\draw[red!90!black] (4.4,0.1) node {~ if $V^*_{x}$ intersects $\frac{\Phi(xM)}{xM}$};
\draw[line width = 0.8, dotted] (4.23,2) -- (4.23,3.75);
\draw (2.3,1.7) node {$\alpha_1$};
\draw (6.8,1.7) node {$\alpha_2$};
\draw (4.23,1.7) node {$x^*$};
\draw (12,2.4) node {$y=V^*_x(x)$};
\end{tikzpicture}
\caption{The only case which doesn't lead to a contradiction is when $V^*_x$ crosses $\frac{\Phi(xM)}{xM}$ only once and the crossing is from above.}
\label{fig:boundarycrossing4}
\end{center}
\end{figure}
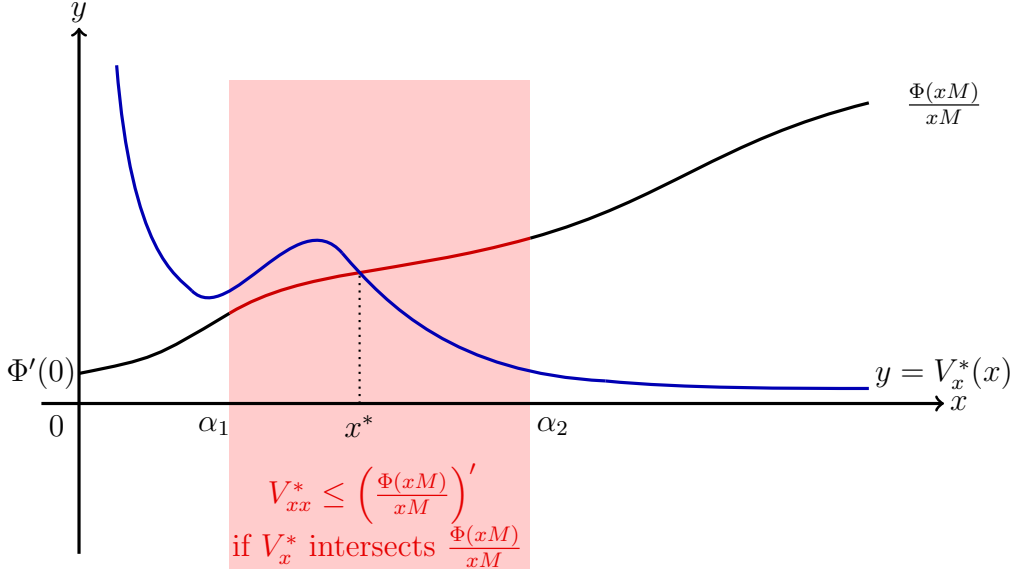
%%%%%%%%%%%%%%%%%
\begin{prop}\label{p:crossing2}
The function $V^*_x$ intersects the curve $\frac{\Phi(xM)}{xM}$ at most three times on $[0,\infty)$.
\end{prop}
\begin{proof}
By \eqref{e:HJB2} if we set $f_x:=\vp$, then at the intersections $x:\; \vp(x) = \frac{\Phi(xM)}{xM}$ we have
\[
\varphi_x = \frac{2}{\sigma^2x^2}\left(\rho-x\mu(x)\frac{\Phi(xM)}{xM}\right),
\]
from the HJB equation. Now we want to compare $\vp_x$ with $\left(\frac{\Phi(xM)}{xM}\right)'$ whenever there is a crossing, to infer the direction from which $\vp$ is crossing. To do that, consider the equation $\vp_x = \left(\frac{\Phi(xM)}{xM}\right)'$. Substituting and simplifying gives us the condition $G(x) + \frac{2M\rho}{\sigma^2} = 0$. Since $G(x)$ has only one extremum by assumption, this equation has zero, one or two solutions. When there are two solutions, say $\alpha_1, \alpha_2$, any intersection of $\vp$ with $\frac{\Phi(xM)}{xM}$ for $x \in (\alpha_1,\alpha_2)$ will have to be with $\vp$ coming from above, as $\vp_x <0$ in that interval.
Using similar arguments to those in Proposition \ref{p:crossing}, this implies, together with the condition on $G$ from \eqref{e:G}, that $\varphi$ can intersect $\frac{\Phi(xM)}{xM}$ at most three times.
\end{proof}
The rest of the proof also mirrors the one of Theorem \ref{t:main}. Apply the four results above and find again that the optimal control is bang-bang with a single threshold $x^*$,
\begin{equation*}
\begin{aligned}
v(x)
&= \begin{cases}
0 & \mbox{if $0< x\leq x^*$}, \\
M & \mbox{if $x>x^*$}.
\end{cases}
\end{aligned}
\end{equation*}
for some $x^*\in (0,\infty)$ (see Figure \ref{fig:boundarycrossing4}).
\end{proof}

%%%%%%%%%%%%%%%%%%%%

\end{document}